\documentclass{amsart}

\usepackage{amsmath}
\usepackage{mathtools}
\usepackage{graphicx}
\usepackage{enumerate}
\usepackage{multirow}
\usepackage{color}
\usepackage{amsfonts}
\usepackage{amssymb}
\usepackage{mathrsfs}
\usepackage{amscd}
\usepackage{url}
\usepackage{tikz-cd}
\usepackage{booktabs}

\newcommand{\re}{\mathbb{R}}

\newcommand{\C}{\mathbb{C}}

\newcommand{\diag}{\mbox{diag}}

\newcommand{\vareps}{\varepsilon}
\newcommand{\dt}{\delta}

\newcommand{\tomg}{\widetilde{\Omega}}
\newcommand{\ta}{\tilde{a}}
\newcommand{\tb}{\tilde{b}}

\newcommand{\tv}{\tilde{v}}
\newcommand{\tu}{\tilde{u}}

\newcommand{\tw}{\tilde{w}}
\newcommand{\tx}{\tilde{x}}
\newcommand{\tk}{\tilde{k}}
\newcommand{\tsig}{\tilde{\sigma}}
\def\af{\alpha}
\def\bt{\beta}
\def\gm{\gamma}
\def\rank{\mbox{rank}}

\newcommand{\sig}{\sigma}
\newcommand{\Sig}{\Sigma}

\newcommand{\mt}[1]{\mathtt{#1}}

\newcommand{\reff}[1]{(\ref{#1})}

\newcommand{\mc}[1]{\mathcal{#1}}

\newcommand{\ddd}{,\ldots,}

\newcommand{\Span}{\mbox{span}}
\newcommand{\Flat}{\mbox{Flat}}

\newcommand{\bdes}{\begin{description}}
\newcommand{\edes}{\end{description}}

\newcommand{\bal}{\begin{align}}
\newcommand{\eal}{\end{align}}

\newcommand{\bnum}{\begin{enumerate}}
\newcommand{\enum}{\end{enumerate}}

\newcommand{\bit}{\begin{itemize}}
\newcommand{\eit}{\end{itemize}}

\newcommand{\bea}{\begin{eqnarray}}
\newcommand{\eea}{\end{eqnarray}}
\newcommand{\be}{\begin{equation}}
\newcommand{\ee}{\end{equation}}

\newcommand{\baray}{\begin{array}}
\newcommand{\earay}{\end{array}}

\newcommand{\bsry}{\begin{subarray}}
\newcommand{\esry}{\end{subarray}}

\newcommand{\bca}{\begin{cases}}
\newcommand{\eca}{\end{cases}}

\newcommand{\bcen}{\begin{center}}
\newcommand{\ecen}{\end{center}}

\newcommand{\bbm}{\begin{bmatrix}}
\newcommand{\ebm}{\end{bmatrix}}

\newcommand{\bmx}{\begin{matrix}}
\newcommand{\emx}{\end{matrix}}

\newcommand{\bpm}{\begin{pmatrix}}
\newcommand{\epm}{\end{pmatrix}}

\newcommand{\btab}{\begin{tabular}}
\newcommand{\etab}{\end{tabular}}

\newtheorem{theorem}{Theorem}[section]

\newtheorem{prop}[theorem]{Proposition}

\newtheorem{defi}[theorem]{Definition}
\theoremstyle{definition}
\newtheorem{example}[theorem]{Example}

\newtheorem{alg}[theorem]{Algorithm}

\newtheorem{remark}[theorem]{Remark}

\numberwithin{equation}{section}

\begin{document}

\title{Rank One Completion for Higher Order Tensors}

\author[Linghao~Zhang]{Linghao Zhang}
\author[Ioana~Dumitriu]{Ioana Dumitriu}
\author[Jiawang~Nie]{Jiawang Nie}

\address{Linghao Zhang, Ioana Dumitriu, Jiawang Nie,
Department of Mathematics, 
University of California San Diego,
9500 Gilman Drive, La Jolla, CA 92093, USA.}
\email{liz010@ucsd.edu, idumitriu@ucsd.edu, njw@math.ucsd.edu}

\date{}

\begin{abstract}
We study the rank one completion problem for tensors of arbitrary orders.
The notion of rank one determinable tensors is introduced.
We explore its properties and propose a recursive algorithm for computing rank one tensor completion. 
This algorithm only requires solving linear systems and computing singular vectors.
In the absence of noise, it produces a unique rank one completion under some assumptions.
In the presence of noise, we show that the computed rank one tensor completion is close to the exact one
when the noise is sufficiently small.
Numerical experiments demonstrate the efficiency and accuracy of the proposed method.
\end{abstract}

\keywords{tensor, completion, extractability, determinability, rank, flattening}

\subjclass[2020]{15A69, 90C23, 65F99}

\maketitle

\section{Introduction}
A tensor $\mc{A}$ of order $m$ (a positive integer) and dimension $(n_1 \ddd n_m)$ can be viewed as the multidimensional array such that
\[
\mc{A} \,=\, ( \mc{A}_{i_1 \cdots i_m} )_{1 \le i_1 \le n_1 \ddd 1 \le i_m \le n_m}.
\]
We denote by $\re^{n_1 \times \cdots \times n_m}$ the space of real tensors as above.
A {\it rank one} tensor is of the form
$u_1\otimes \cdots \otimes u_m $ for vectors $u_1 \in \re^{n_1}  \ddd u_m \in \re^{n_m}$ such that
\[
(u_1 \otimes \cdots \otimes u_m)_{i_1 \cdots i_m} \,=\, (u_1)_{i_1} \cdots (u_m)_{i_m}.
\]
The {\it rank} of $\mc{A}$ is the smallest number of rank one tensors whose sum equals $\mc{A}$.
This is referred to as the Candecomp-Parafac (CP) decomposition in the literature \cite{hitchcock1927,hitchcock1928,Lim13}. 
There are various notions of tensor ranks (e.g., border rank, multilinear rank),
and tensor ranks depend on the ground field.
Moreover, determining tensor ranks is NP-hard \cite{Hillar2013}.
We refer to \cite{LMV2000,harshman1970,hitchcock1927,hitchcock1928,Lim13,nieGP,nieSymApprox} for related work on tensor decompositions and tensor ranks.

The {\it tensor completion problem} (TCP) 
is to fill in missing entries of a partially observed tensor
such that the completed tensor has some properties (e.g., it has a low rank).
TCPs appear in a broad range of applications, 
including recommendation systems \cite{Frolov2015,Karatzoglou2010}, 
computer vision \cite{qin2022,qiu2021}, 
imaging and signal processing \cite{Lim2010,Lim2013,zhao2020};
see also \cite{KolBad09} for more applications.
In contrast to matrix completion problems, TCP is substantially more challenging.
There exists a variety of methods for low rank tensor completions,
including sum-of-squares relaxations \cite{barak2022},
tensor nuclear norm relaxations
\cite{mu2014,nieNuclear,tang2015,tian2024,yuan2016},
approaches based on CP and Tucker decompositions
\cite{ashraphijuo2017,bai2016,qiu2021,zhao2020},
Riemannian-manifold optimization
\cite{dong2022,kressner2014,steinlechner2016,swijsen2022},
and methods based on tubal rank and t-SVD
\cite{jiang2019,kilmer2011,wang19,zhang2017}.

This paper focuses on the special but fundamental case of rank one tensor completion.
Solving this problem sheds light on more general low rank completion problems,
as rank one tensors exhibit rich algebraic and geometric problems.
We consider an $m$th order partially observed tensor (p.o.t.)
$\mc{A}$ with the observation label set
\[
\Omega \subseteq [n_1] \times \cdots \times [n_m]
\quad \text{where} \quad [n_k] \coloneqq \{ 1\ddd n_k \},
\]
such that $\mc{A}_{i_1 \cdots i_m}$ is observed (given) for all $(i_1 \ddd i_m) \in \Omega$.
For convenience, we denote this p.o.t. by $(\mc{A}, \Omega)$.
The rank one tensor completion problem is to look for vectors
$u_1 \ddd u_m$ such that 
\be\label{tc_comp}
\mc{A}_{i_1 \cdots i_m} \,=\, (u_1)_{i_1} \cdots (u_m)_{i_m}
\quad \text{for all} \quad (i_1\ddd i_m) \in \Omega.
\ee
If such $u_1 \ddd u_m$ exist, $(\mc{A}, \Omega)$ is said to be {\it rank one completable},
and $u_1 \otimes \cdots \otimes u_m$ is called a {\it rank one completion} for $(\mc{A}, \Omega)$.
Moreover, $(\mc{A}, \Omega)$ is said to have a {\it unique} rank one completion if
for any two rank one completions 
$u_1 \otimes \cdots \otimes u_m$ and $v_1 \otimes \cdots \otimes v_m$
of $(\mc{A}, \Omega)$,
there exist nonzero scalars $\af_k$ such that 
$u_k = \af_k v_k$ for $k=1 \ddd m$.

The rank one tensor completion problem has been studied from several perspectives in the literature.
Kahle et al. \cite{kahle17} investigate its algebraic and geometric properties.
Some computational methods are proposed in \cite{cifuentes25,nie25,Singh20,zhou25}.
The work \cite{Singh20} studies uniqueness for rank one tensor completion,
and \cite{cifuentes25} gives a semidefinite programming (SDP) approach.
The work \cite{zhou25} is motivated by the following observation for cubic order tensors:
the rank one tensor $a \otimes b \otimes c$ is a completion for $(\mc{A}, \Omega)$ 
if and only if the vectors $a$ and $b$ satisfy
\be\label{det}
\det \bbm 
\mc{A}_{ijk} & a_ib_j \\
\mc{A}_{i'j'k} & a_{i'}b_{j'} \\
\ebm = 0
\ee
for all $(i,j,k), (i',j',k) \in \Omega$.
Equivalently, if we let $X = ab^T$, then 
\be\label{detX}
\mc{A}_{ijk} X_{i'j'} - \mc{A}_{i'j'k} X_{ij} = 0
\quad \text{for all} \quad (i,j,k), (i',j',k) \in \Omega.
\ee
This gives a system of linear equations in
$
X_{\bar{\Omega}} \coloneqq (X_{ij})_{(i,j) \in \bar{\Omega}},
$
where $\bar{\Omega}$ denotes the projection of $\Omega$ onto the first two indices.
When the solution space of \reff{detX} is one-dimensional, 
the partial matrix $X_{\bar{\Omega}}$ can be determined uniquely up to scaling.
If $(\mc{A}, \Omega)$ is rank one completable, 
then $X_{\bar{\Omega}}$ can be completed to a full matrix $X$ of rank one.
Then, we can obtain vectors $a, b$ (up to scaling) from $X$,
and thereby $c$ from $\mc{A} = a \otimes b \otimes c$ on $\Omega$.
For such a case, $(\mc{A}, \Omega)$ is said to be strongly rank one completable in \cite{zhou25}.
However, this approach is specific to cubic order tensors 
and does not extend naturally to the higher order case.
Moreover, it is not applicable in the presence of noise.
The follow-up work \cite{nie25} considers noisy rank one completion for cubic order tensors.
It formulates the problem as biquadratic optimization with sphere constraints and solves it by an SDP relaxation.
This approach can be effective in practice, 
but the biquadratic optimization may not always be solved globally by that relaxation.

\subsection*{Contributions}
In this paper, we propose an efficient approach for completing rank one tensors of arbitrary orders.
Let $(\mc{A}, \Omega)$ be a p.o.t. of order $m$.
Suppose it admits a rank one completion
$
\mc{A} = u_1 \otimes \cdots \otimes u_m.
$
The tensor $\mc{A}$ can be flattened into a matrix in various ways.
For convenience of description, we flatten $\mc{A}$ along the first $m-1$ indices,
so it becomes a rank one matrix
$
A = ab^T,
$
where $a$ is the vectorization of the tensor product $u_1 \otimes \cdots \otimes u_{m-1}$ and $b = u_m$.
In this way, the p.o.t. $(\mc{A}, \Omega)$ induces a partially observed matrix $(A, \Omega)$.
As shown in Section~\ref{sc:ext_mx},
the vector $a$ must satisfy a homogeneous linear system
\be\label{linsym:Bx=0}
Bx \,=\, 0,
\ee
where $B$ depends on $(\mc{A}, \Omega)$ but not on $b$.
Assume the solution space of \reff{linsym:Bx=0} is one-dimensional, 
then $a$ can be determined uniquely up to scaling. 
When the observed tensor entries are exact, 
we can recover $a$ from a nontrivial solution $x^*$ to \reff{linsym:Bx=0}. 
When these entries have noise, \reff{linsym:Bx=0} may not have a nontrivial solution,
but we can instead select $x^*$ as the right singular vector corresponding to the smallest singular value of $B$.
Once $x^*$ is computed, the vector $b$ can be recovered from the linear equation
$
A_{ij} = x^*_i b_j
$
on $\Omega$;
that is, $u_m$ is obtained.
Next, we reshape $x^*$ to a p.o.t. of order $m-1$, say, $(\mc{A}_1, \Omega_1)$.
We apply the above procedure to $(\mc{A}_1, \Omega_1)$,
and this yields the vector $u_{m-1}$.
Repeating this process recursively, we obtain vectors $u_{m-2} \ddd u_1$, 
which give a rank one completion.

Our new method has several advantages.
First, it applies to tensors of arbitrary orders.
Second, it remains robust in the presence of noise.
Third, our method only requires solving linear systems and computing singular vectors,
hence it is very efficient for large-scale problems.
Our main contributions are:

\bit
\item We introduce the notion of rank one {\it extractability} for partially observed matrices
and study its properties.

\item We extend the rank one matrix extractability to higher order tensors
by introducing the notion of rank one {\it determinability}.
The properties of rank one determinable tensors are studied.

\item Based on the above,
we propose a recursive algorithm (Algorithm~\ref{alg:recursive}) for completing rank one tensors of order $m$.
It produces the chain
\[
(\mc{A}_0, \Omega_0) 
\longrightarrow(\mc{A}_1, \Omega_1) 
\longrightarrow (\mc{A}_2, \Omega_2)
\longrightarrow
\cdots 
\longrightarrow (\mc{A}_{m-1}, \Omega_{m-1}),
\]
where each $\mc{A}_j$ is a p.o.t. of order $m-j$
(see Section~\ref{sc:TC_alg}).

\item Our algorithm only requires solving linear systems and computing singular vectors,
and we prove that it is robust under noise.
\eit

The paper is organized as follows.
Section~\ref{sc:ext_mx} introduces rank one extractable matrices
and studies the properties.
In Section~\ref{sc:det_tensor}, 
we define rank one determinability through tensor flattenings
and explore its connection with rank one completion.
Section~\ref{sc:TC_alg} proposes a recursive algorithm for completing rank one tensors of arbitrary orders.
Section~\ref{sc:perturb} gives perturbation analysis under noise.
Numerical experiments are presented in Section~\ref{sc:num_exp}.

\subsection*{Notation}
Denote by $\re$ (resp., $\C$) the real (resp., complex) field.
For a positive integer $m$, denote $[m] \coloneqq \{1\ddd m\}$.
For a set $S$, we denote its cardinality as $|S|$.
Let $e_i$ be the canonical basis vector such that the $i$th entry is $1$ and $0$ otherwise.
For a vector $x \in \re^n$, its Euclidean norm is
$\|x\| \coloneqq \sqrt{x^Tx}$.
For a matrix $A \in \re^{n_1 \times n_2}$, 
its spectral norm, denoted by $\|A\|$, is the largest singular value of $A$.
Denote by $\text{null}(A)$ the nullspace of $A$.
For a p.o.t. $(\mc{A}, \Omega)$, we define the norm
\[
\baray{c}
\|\mc{A}\|_{\Omega} 
\,\coloneqq\, \Big( \sum\limits_{(i_1 \ddd i_m) \in \Omega} |\mc{A}_{i_1 \cdots i_m}|^2 \Big)^{1/2}.
\earay
\]

\section{Rank one extractable matrices}\label{sc:ext_mx}

For a partially observed matrix (p.o.m.) $A \coloneqq (A_{ij}) \in \re^{n_1 \times n_2}$,
let $\Omega \subseteq [n_1] \times [n_2]$ be the set of indices $(i,j)$ 
such that the entry $A_{ij}$ is observed.
Denote
\be\label{S1S2}
S_1 = \{ i : (i,j) \in \Omega \}, 
\quad S_2 = \{ j : (i,j) \in \Omega \}.
\ee
We write $\Omega = \cup_{j \in S_2} \Theta_j$ where 
\be\label{theta_j}
\Theta_j = \{ (i,j) : (i,j) \in \Omega \} = \{ (i_1, j), (i_2, j) \ddd (i_{m_j}, j) \},
\ee
where $m_j = |\Theta_j|$ is the cardinality of $\Theta_j$.
If $(A, \Omega)$ is rank one completable, there exist $x \in \re^{n_1}, y \in \re^{n_2}$ 
such that $A_{ij} = x_iy_j$ for all $(i,j) \in \Omega$.
Then,
\[
\bbm A_{i_1 j} \\ 
\vdots \\ 
A_{i_{m_j} j} 
\ebm
= \bbm x_{i_1} \\ 
\vdots \\ 
x_{i_{m_j}} 
\ebm y_j
\quad \implies \quad
\rank 
\bbm 
A_{i_1 j} & x_{i_1} \\ 
\vdots & \vdots \\ 
A_{i_{m_j} j} & x_{i_{m_j}} 
\ebm \le 1.
\]
This implies the 2-by-2 minor
$
A_{i_s j} x_{i_t} - A_{i_t j} x_{i_s} = 0
$
for all $1 \le s < t \le m_j$.
This gives 
$\binom{m_j}{2}$
linear equations for each $j \in S_2$,
and there are
$
\sum_{j \in S_2} \binom{m_j}{2}
$
equations in total.
These equations are homogeneous in 
\[
x \, \coloneqq \, ( x_i )_{i \in S_1}.
\]
The set of solutions to these linear equations is a subspace of $\re^{S_1}$.
We are interested in the case that the solution space is one-dimensional.

\begin{defi}\label{def:rank1_det_mx}
A p.o.m. $(A, \Omega)$ is said to be rank one extractable if the solution space of 
the following system in $x \in \re^{S_1}$ is one-dimensional:
\be\label{eq:mx_det}
A_{i j} x_{i'} - A_{i' j} x_{i} = 0, \quad (i, j), (i', j) \in \Omega, \quad i < i', \quad j \in S_2.
\ee
\end{defi}

Observe that the system \reff{eq:mx_det} can be equivalently written as
\be\label{Bx=0_mx}
Bx \,=\, 0
\ee
for some sparse matrix $B$,
each row of which has at most two nonzero entries.

\begin{example}
Consider the p.o.m. $A \in \re^{3 \times 3}$ with observed entries:
\[
A_{11} = 2, \quad A_{12} = -4, \quad A_{13} = 6, \quad A_{21} = 3, \quad A_{22} = -6, \quad A_{31} = -1.
\]
We have $\Omega = \{(1,1), (1,2), (1,3), (2,1), (2,2), (3,1)\}$ and 
$S_1 = S_2 = \{1,2,3\}$.
The nullspace of $B$ as in \reff{Bx=0_mx} is one-dimensional,
so $(A, \Omega)$ is rank one extractable.
\end{example}

For a subset $\Omega \subseteq [n_1] \times [n_2]$,
let $G(V_1, V_2, \Omega)$ be the bipartite graph with vertex sets 
$V_1= \{ 1 \ddd n_1 \}$, $V_2 = \{ 1 \ddd n_2 \}$ and the edge set $\Omega$.
For convenience, we say a vector is {\it entirely nonzero} if all its entries are nonzero.

\begin{theorem}\label{thm:det_comp}
Let $(A, \Omega)$ be a p.o.m. such that $A_{ij} \ne 0$ for all $(i,j) \in \Omega$.
\bnum

\item[(i)] Suppose $(A, \Omega)$ is rank one completable. 
Then $(A, \Omega)$ has a unique rank one completion if and only if the graph $G(V_1, V_2, \Omega)$ is connected.

\item[(ii)] If \reff{eq:mx_det} has an entirely nonzero solution $\hat{x} \in \re^{S_1}$, then 
there exists an extension $a \in \re^{n_1}$ of $\hat{x}$ (i.e., $\hat{x}$ is a subvector of $a$) such that $A = ab^T$ for some $b \in \re^{n_2}$, where all entries of $a,b$ are nonzero.

\item[(iii)] The p.o.m. $(A, \Omega)$ is rank one completable if and only if \reff{eq:mx_det} has an entirely nonzero solution.

\item[(iv)] Suppose the graph $G(V_1, V_2, \Omega)$ is connected.
Then every nontrivial solution to \reff{eq:mx_det} is entirely nonzero.
If, in addition, $(A, \Omega)$ is rank one completable, then $(A, \Omega)$ is rank one extractable.

\enum
\end{theorem}

\begin{proof}
(i) This conclusion can be found in \cite[Lemma~1]{cosse2021} and \cite{kiraly12}.

\noindent (ii) Let $\hat{x}$ be an entirely nonzero solution to \reff{eq:mx_det}.
For each $j \in S_2$, \reff{eq:mx_det} implies
\[
\rank 
\bbm
A_{i_1 j} & \hat{x}_{i_1} \\
\vdots & \vdots \\
A_{i_{m_j} j} & \hat{x}_{i_{m_j}} \\
\ebm = 1
\quad \implies \quad 
\bbm
A_{i_1 j} \\
\vdots \\
A_{i_{m_j} j} \\
\ebm
= \beta_j \bbm
\hat{x}_{i_1} \\
\vdots \\
\hat{x}_{i_{m_j}} \\
\ebm, \,
0 \ne \beta_j \in \re.
\]
Let $a \in \re^{n_1}$ and $b \in \re^{n_2}$ be such that
\[
\begin{gathered}
a_i = \hat{x}_i \quad \text{for }  i \in S_1, \qquad
b_j = \beta_j \quad \text{for } j \in S_2, \\
a_i = 1 \quad \text{for } i \in [n_1] \setminus S_1, \qquad
b_j = 1 \quad \text{for } j \in [n_2] \setminus S_2.
\end{gathered}
\]
Then $A_{ij} = \hat{x}_i \beta_j = a_ib_j$ on $\Omega$.
Hence, $A = ab^T$.

\noindent (iii) The ``if'' direction is shown in (ii).
Conversely,
if $(A, \Omega)$ is rank one completable, 
then there exist $a \in \re^{n_1}$ and $b \in \re^{n_2}$ such that $A_{ij} = a_ib_j$ on $\Omega$.
Let $\hat{x} \in \re^{S_1}$ be a subvector of $a$ satisfying $\hat{x}_i = a_i$ for all $i \in S_1$.
Then, for each $j \in S_2$, $(i,j), (i', j) \in \Omega$, $i < i'$, we have
\[
A_{i j} \hat{x}_{i'} - A_{i' j} \hat{x}_{i} = ( a_ia_{i'} - a_{i'}a_i ) b_j = 0.
\]
So, $\hat{x}$ solves \reff{eq:mx_det} and is entirely nonzero since $A_{ij} = a_ib_j \ne 0$ for all $(i,j) \in \Omega$.

\noindent (iv) Let $x$ be a nontrivial solution to \reff{eq:mx_det}.
For each $j \in S_2$, \reff{eq:mx_det} implies
\[
\rank 
\bbm
A_{i_1 j} & x_{i_1} \\
\vdots & \vdots \\
A_{i_{m_j} j} & x_{i_{m_j}} \\
\ebm = 1
\quad \implies \quad
\bbm
x_{i_1} \\
\vdots \\
x_{i_{m_j}} \\
\ebm =
\af_j \bbm
A_{i_1 j} \\
\vdots \\
A_{i_{m_j} j} \\
\ebm, \,
\af_j \in \re.
\]
This means that for the above $j$, $x_i = \af_j A_{ij}$ for all $i$ with $(i,j) \in \Omega$.
Suppose otherwise $x$ has a zero entry, say
$x_{i_1} = 0$ for some $i_1 \in S_1$.
Then, for every $j$ with $(i_1, j) \in \Omega$, we have
\be\label{af_j=0}
0 = x_{i_1} = \af_j A_{i_1 j} \implies \af_j = 0
\ee
since $A_{i_1 j} \ne 0$.
Since $G(V_1, V_2, \Omega)$ is connected, there exists a path
\[
i_1 \to j_1 \to i_2 \to j_2 \to i_3 \to j_3 \to \cdots \to j_{N-1} \to i_{N}
\]
for some $N \ge |S_1|$ such that $S_1 \subseteq \{ i_1 , i_2 \ddd i_N \}$ and 
\[
(i_k, j_k), (i_{k+1}, j_k) \in \Omega, \quad k = 1 \ddd N-1.
\]
We remark that $i_k, j_k$ may be repeated.
By \reff{af_j=0}, we have
$0 = x_{i_1} = \af_{j_1} A_{i_1 j_1}$,
which implies $\af_{j_1} = 0$.
Since $(i_2, j_1) \in \Omega$, we get $x_{i_2} = \af_{j_1} A_{i_2 j_1} = 0$.
Repeating this argument, we can similarly get
$x_{i_3} = \cdots = x_{i_N} = 0$.
Hence, $x = 0$, which contradicts that $x$ is a nontrivial solution to \reff{eq:mx_det}.
Therefore, every nontrivial solution to \reff{eq:mx_det} is entirely nonzero.

Moreover, if $(A, \Omega)$ is rank one completable, then by (iii),
\reff{eq:mx_det} has an entirely nonzero solution $\hat{x} \in \re^{S_1}$.
Suppose there exists another nontrivial solution $\tx \in \re^{S_1}$ to \reff{eq:mx_det}.
By the first part of (iv), $\tx$ is entirely nonzero.
So by (ii), there exists an extension of $\hat{x}$ (resp., $\tx$), 
say $\hat{a} \in \re^{n_1}$ (resp., $\ta \in \re^{n_1}$), 
such that $A_{ij} = \hat{a}_i\hat{b}_j = \ta_i\tb_j$ on $\Omega$, 
for some $\hat{b}, \tb \in \re^{n_2}$.
Since $G(V_1, V_2, \Omega)$ is connected, (i) implies that $(A, \Omega)$ has a unique rank one completion.
Then there exists a scalar $\tau \ne 0$ such that 
$\ta = \tau \hat{a}$ and $\tb = \tau^{-1} \hat{b}$.
Hence, $\tx = \tau \hat{x}$.
This means \reff{eq:mx_det} has a unique nontrivial solution up to scaling,
i.e., $(A, \Omega)$ is rank one extractable.
\end{proof}

We make the following remarks:
\bit

\item If $(A, \Omega)$ is rank one extractable, then $(A, \Omega)$ may not be rank one completable.
Consider the p.o.m. $A \in \re^{2 \times 2}$ with observed entries:
$A_{11} = 1$, $A_{21} = 0$, $A_{22} = 1$.
The solution space of \reff{eq:mx_det} is one-dimensional,
so $(A, \Omega)$ is rank one extractable,
but $(A, \Omega)$ is not rank one completable.

\item If $(A, \Omega)$ is rank one extractable, then $(A, \Omega)$ may not have a unique rank one completion
even if it is rank one completable.
Consider the p.o.m. $A \in \re^{2 \times 2}$ with observed entries:
$A_{11} = 1$, $A_{21} = 2$.
The solution space of \reff{eq:mx_det} is one-dimensional.
This means $(A, \Omega)$ is rank one extractable.
However, $(A, \Omega)$ has infinitely many rank one completions.
For instance, $A = ab^T$ with $a = (1,2)$ and $b = (1,t)$ for any $t \in \re$.

\item The rank one extractability does not imply connectedness of $G(V_1, V_2, \Omega)$ 
or vice versa.
As shown in the above item, $(A, \Omega)$ is rank one extractable, 
but $G(V_1,V_2,\Omega)$ is not connected (the vertex $2 \in V_2$ is isolated).
Conversely, consider the p.o.m. $A \in \re^{3 \times 3}$ with observed entries:
\[
\baray{cccccc}
A_{11} = 1, & A_{12} = 1, & A_{21} = 2, & A_{22} = 3, & A_{23} = 4, & A_{32} = 5.
\earay
\]
Then, $\Omega = \{ (1,1), (1,2), (2,1), (2,2), (2,3), (3,2) \}$
and $G(V_1,V_2,\Omega)$ is connected.
However, one can check that the solution space of \reff{eq:mx_det} is zero-dimensional, 
i.e., $(A, \Omega)$ is not rank one extractable.
\eit

The above shows that there is no direct implication among rank one completability,
rank one extractability, and the connectedness of $G(V_1, V_2, \Omega)$.
However, they are equivalent under some assumptions.

\begin{defi}
For $\Omega \subseteq [n_1] \times [n_2]$, we say that $\Omega$
is row-full if $S_1 = [n_1]$, and $\Omega$ is column-full if $S_2 = [n_2]$.
\end{defi}

We remark that 
$\Omega$ is row-full (resp., column full) if and only if there is no isolated vertex in $V_1$ (resp., $V_2$)
in the bipartite graph $G(V_1, V_2, \Omega)$.

\begin{theorem}\label{thm:connect<->rank-1}
Let $(A, \Omega)$ be a p.o.m. such that $A_{ij} \ne 0$ for all $(i,j) \in \Omega$.
\bnum

\item[(i)] Suppose the graph $G(V_1, V_2, \Omega)$ is connected. 
Then, $(A, \Omega)$ is rank one extractable if and only if $(A, \Omega)$ is rank one completable.

\item[(ii)] Suppose $\Omega$ is row-full and column-full, and
$(A, \Omega)$ is rank one completable. 
Then, the following are equivalent:
\bit
\item[(a)]  $(A, \Omega)$ is rank one extractable;

\item[(b)] $(A, \Omega)$ has a unique rank one completion;

\item[(c)] the graph $G(V_1, V_2, \Omega)$ is connected.
\eit

\enum
\end{theorem}

\begin{proof}
(i) Suppose $G(V_1, V_2, \Omega)$ is connected.
If $(A, \Omega)$ is rank one extractable, then \reff{eq:mx_det} has a unique nontrivial solution 
(up to scaling),
say, $\hat{x}$.
Then Theorem~\ref{thm:det_comp}(iv) implies that $\hat{x}$ is entirely nonzero.
By Theorem~\ref{thm:det_comp}(ii), $(A, \Omega)$ is rank one completable.
Conversely, if $(A, \Omega)$ is rank one completable, then $(A, \Omega)$ is rank one extractable by Theorem~\ref{thm:det_comp}(iv).

\noindent (ii) $(a) \Rightarrow (b)$: 
Let $A = ab^T$ be a rank one completion.
Since $(A, \Omega)$ is rank one extractable and $\Omega$ is row-full,
$a$ is the unique (up to scaling) solution to \reff{eq:mx_det}.
Then, $b$ is also unique since $\Omega$ is column-full.
So, $(A, \Omega)$ has a unique rank one completion.

\noindent $(b) \Rightarrow (c)$: 
Since $(A, \Omega)$ has a unique rank one completion, 
$G(V_1, V_2, \Omega)$ is connected by Theorem~\ref{thm:det_comp}(i). 

\noindent $(c) \Rightarrow (a)$: This is implied by item~(i).
\end{proof}

We remark that in Definition~\ref{def:rank1_det_mx}, 
the rank one extractability of $(A, \Omega)$ is defined by enumeration over the index $j$.
Similarly, one can also consider enumeration over the index $i$.
This is related to the rank one extractability of $(A^T, \Omega')$ for 
\be\label{omg'}
\Omega' = \{ (j,i) : (i,j) \in \Omega \}, 
\quad S_1' = S_2, \quad S_2' = S_1.
\ee
However, the rank one extractability of $(A, \Omega)$ does not imply 
the rank one extractability of $(A^T, \Omega')$.
This is shown by the following example.

\begin{example}
Consider the p.o.m. $A \in \re^{3 \times 3}$ with observed entries:
$A_{11} =  A_{12} = 1$, $A_{21} = A_{23} = A_{31} = 0$.
Then $\Omega = \{(1,1), (1,2), (2,1), (2,3), (3,1)\}$.
The solution space of \reff{eq:mx_det} is one-dimensional, so 
$(A, \Omega)$ is rank one extractable.
However, the solution space of \reff{eq:mx_det} for $(A^T, \Omega')$ 
is two-dimensional, so $(A^T, \Omega')$ is not rank one extractable.
\end{example}

\begin{prop}\label{prop:extra_transpose}
Let $(A, \Omega)$ be a p.o.m. such that $A_{ij} \ne 0$ for all $(i,j) \in \Omega$. 
\bnum

\item[(i)] Suppose $(A, \Omega)$ is rank one completable.
Then, $(A , \Omega)$ is rank one extractable if and only if $(A^T, \Omega')$ is rank one extractable.

\item[(ii)] Suppose the graph $G(V_1, V_2, \Omega)$ is connected.
Then, $(A , \Omega)$ is rank one extractable if and only if $(A^T, \Omega')$ is rank one extractable.
\enum
\end{prop}

\begin{proof}
(i) Without loss of generality, we can assume $\Omega$ is row-full and column-full.
For the ``only if'' direction, we know the graph $G(V_1, V_2, \Omega)$ is connected by Theorem~\ref{thm:connect<->rank-1}(ii).
This implies $G(V_2, V_1, \Omega')$ is also connected.
Since $(A, \Omega)$ is rank one completable, $(A^T, \Omega')$ is also rank one completable.
By Theorem~\ref{thm:connect<->rank-1}(i), $(A^T, \Omega')$ is rank one extractable.
The proof for the ``if'' direction is the same.

\noindent (ii) Since $G(V_1, V_2, \Omega)$ is connected,
by Theorem~\ref{thm:connect<->rank-1}(i),
we know
$(A, \Omega)$ is rank one extractable if and only if $(A, \Omega)$ is rank one completable.
Note that $G(V_2, V_1, \Omega')$ is also connected,
and $(A, \Omega)$ being rank one completable is equivalent to $(A^T, \Omega')$ being rank one completable.
So, the conclusion follows from item~(i).
\end{proof}

\section{Rank one determinable tensors}\label{sc:det_tensor}

The notion of rank one extractable matrices in Section~\ref{sc:ext_mx} can be extended to tensors.
For an integer $k \in [m]$, denote 
\begin{eqnarray}
\label{vec_i}
\vec{i}  & \coloneqq & \big( i_1 \ddd i_{k-1}, i_{k+1} \ddd i_m \big),   \\
\label{I_k}
I_k  & \coloneqq & \big\{ \vec{i} : 1 \le i_l \le n_l \text{ for }l \ne k \big\}.
\end{eqnarray}
Denote by $\re^{ I_k \times [n_k] }$ the space of matrices with rows indexed by $\vec{i} \in I_k$
and columns indexed by $j \in [n_k]$.
For $\mc{A} \in \re^{n_1 \times \cdots \times n_m}$, 
let $\Flat_k(\mc{A}) \in \re^{ I_k \times [n_k] }$ be such that
\be\label{flat_k}
\big( \Flat_k(\mc{A}) \big)_{\vec{i} j} \,\coloneqq\, \mc{A}_{i_1 \cdots i_{k-1} j i_{k+1} \cdots i_m}
\ee
with $\vec{i} \in I_k$ and $j \in [n_k]$.
The matrix $\Flat_k(\mc{A})$ is called the $k$th {\it flattening matrix} of $\mc{A}$.
We refer to \cite{KolBad09} for tensor flattening (or matricization).

For $\Omega \subseteq [n_1] \times \cdots \times [n_m]$,
define the subset of $I_k \times [n_k]$:
\be\label{Omg^k}
\Omega^{(k)} \,\coloneqq\, \{ ( \vec{i}, j ) : 
( i_1 \ddd i_{k-1}, j, i_{k+1} \ddd i_m ) \in \Omega \}.
\ee
When $\mc{A}$ is partially observed on $\Omega$, 
the matrix $\Flat_k(\mc{A})$ is partially observed on $\Omega^{(k)}$.
Denote the index sets
\be\label{S1kS2k}
S^{(k)}_1 = \big\{ \vec{i} : (\vec{i},j) \in \Omega^{(k)} \big\}, 
\quad S^{(k)}_2 = \big\{ j : (\vec{i},j) \in \Omega^{(k)} \big\}.
\ee
Denote the vector variable
\[
x \,\coloneqq\, ( x_{\vec{i}} )_{\vec{i} \in S_1^{(k)}}.
\]
Like \reff{eq:mx_det},
we consider the linear system
\be\label{eq:flat_k_det}
\big(\Flat_k(\mc A)\big)_{\vec{i} j} \, x_{\vec{i'} } - \big(\Flat_k(\mc A)\big)_{\vec{i'} j} \, x_{\vec{i}} \,=\, 0
\ee
for all 
$(\vec{i}, j), (\vec{i'}, j) \in \Omega^{(k)}$, $\vec{i} <_{lex} \vec{i'}$, $j \in S^{(k)}_2$.
Here, $<_{lex}$ denotes the standard lexicographical ordering.
Equivalently, \reff{eq:flat_k_det} reads as
\[
\mc A_{i_{1}\cdots i_{k-1} j i_{k+1}\cdots i_{m}} \, x_{\vec{i'}}
-
\mc A_{i'_{1}\cdots i'_{k-1} j i'_{k+1}\cdots i'_{m}} \, x_{\vec{i}} \,=\, 0,
\]
where $\vec{i'} = ( i'_1 \ddd i'_{k-1}, i'_{k+1} \ddd i'_m )$.
The system \reff{eq:flat_k_det} is equivalent to
\be\label{Bx=0}
B^{(k)}  x \,=\, 0
\ee
for some matrix $B^{(k)}$.
Denote the label set
\be\label{P_k}
\mc{P}_k \,\coloneqq\, [n_1] \times \cdots \times [n_{k-1}] \times [n_{k+1}] \times \cdots \times [n_m].
\ee
For $x \in \re^{S_1^{(k)}}$, 
let $\mc{T}_k(x)$ be the p.o.t. in 
$\re^{\mc{P}_k}$
such that
\be\label{up_k}
\mc{T}_k(x)_{i_1 \cdots i_{k-1} i_{k+1} \cdots i_m} \,=\, x_{\vec{i}}
\ee
for all $\vec{i} \in S^{(k)}_1$.
Define the reshaping map 
\be\label{gm_k}
\Gamma_k : \re^{I_k} \to \re^{\mc{P}_k},
\qquad u \mapsto \Gamma_k (u)
\ee
such that for all $\vec{i} \in I_k$ as in \reff{vec_i},
\[
\Gamma_k (u)_{i_1 \cdots i_{k-1} i_{k+1} \cdots i_m}
\,=\, u_{\vec{i}}.
\]
The above $\Gamma_k (u)$ is a full tensor of order $m-1$.

\begin{defi}\label{def:rank1-det}
The $m$th order p.o.t. $(\mc{A}, \Omega)$ is rank one determinable if:
\bnum

\item[(i)] When $m=2$, $(\mc{A}, \Omega)$ or $(\mc{A}^T, \Omega')$ is rank one extractable, 
where $\Omega'$ is defined as in \reff{omg'}.

\item[(ii)] When $m \ge 3$, there exists $k \in [m]$ such that
$\big( \Flat_k(\mc{A}), \Omega^{(k)} \big)$ is rank one extractable,
and the unique (up to scaling) nontrivial solution $x^*$ to \reff{eq:flat_k_det} satisfies that
$\big( \mc{T}_k(x^*), S_1^{(k)} \big)$ is rank one determinable.
\enum
\end{defi}

In the above, $\mc{T}_k(x^*)$ is a p.o.t. of order $(m-1)$ partially observed on $S_1^{(k)}$.
Now we give some examples of rank one determinable tensors.

\begin{example}\cite{zhou25}
Consider the p.o.t. $\mc{A} \in \re^{3 \times 3 \times 3}$ with observed entries:
\[
\baray{llll}
\mc{A}_{111} = -1, &
\mc{A}_{221} = -1, & 
\mc{A}_{311} = -1, & 
\mc{A}_{132} = 1,  \\
\mc{A}_{312} = -1, & 
\mc{A}_{233} = 1, & 
\mc{A}_{313} = 1, & 
\mc{A}_{323} = -1.
\earay
\]
For $k = 3$, we have $S_1^{(3)} = \{ (1,1), (1,3), (2,2), (2,3), (3,1), (3,2) \}$ and
\[
B^{(3)} =
\left[
\begin{array}{rrrrrr}
1 & 0 & -1  & 0 & 0 & 0\\
1  & 0 & 0 & 0 & -1 & 0\\
0 & 0 & 1  & 0 & -1  & 0\\
0 & 1 & 0 & 0 & 1 & 0\\
0 & 0 & 0 & -1  & 1 & 0\\
0 & 0 & 0 & 1  & 0 & 1 \\
0 & 0 & 0 & 0 & 1 & 1 
\end{array}
\right],
\]
where $B^{(3)}$ is defined as in \reff{Bx=0}.
The solution space of \reff{eq:flat_k_det} is spanned by
$
x^* = (-1, 1, -1, -1, -1, 1).
$
This means $(\Flat_3(\mc{A}), \Omega^{(3)})$ is rank one extractable. 
Note that $A \coloneqq \mc{T}_3(x^*) \in \re^{3 \times 3}$ is the p.o.m. with observed entries:
\[
A_{11} = -1, \ A_{13} = 1, \ A_{22} = -1, \
A_{23} = -1, \ A_{31} = -1, \ A_{32} = 1.
\]
One can check that the solution space of \reff{eq:mx_det} is spanned by $(1, -1, 1)$.
So, $(A, S_1^{(3)})$ is rank one extractable.
Therefore, $(\mc{A}, \Omega)$ is rank one determinable.
It is worth noting that the solution space of \reff{eq:flat_k_det} is two-dimensional for $k=1,2$.
\end{example}

\begin{example}
Consider the p.o.t. $\mc{A} \in \re^{2 \times 2 \times 3 \times 5  \times 8 \times 9}$ with observed entries:
\[
\baray{l}
\mc{A}_{1     2     3     1     3     4}
= \mc{A}_{1     2     3     1     3     6}
= \mc{A}_{1     2     3     1     4     7}
= \mc{A}_{1     2     3     1     4     8}
= \mc{A}_{1     2     3     3     6     3}
= \mc{A}_{1     2     3     3     6     6}
= \mc{A}_{2     1     1     1     4     7} = 1,  \\
\mc{A}_{2     1     1     1     4     9}
= \mc{A}_{2     1     1     1     8     1}
= \mc{A}_{2     1     1     1     8     5}
= \mc{A}_{2     1     1     2     2     4}
= \mc{A}_{2     1     1     2     2     9}
= \mc{A}_{2     1     1     2     5     5}
= \mc{A}_{2     1     1     2     5     6} = 1,\\
\mc{A}_{2     1     2     4     1     4}
= \mc{A}_{2     1     2     4     1     8}
= \mc{A}_{2     1     2     4     6     3}
= \mc{A}_{2     1     2     4     6     8}
= \mc{A}_{2     1     2     5     5     3}
= \mc{A}_{2     1     2     5     5     5}
= \mc{A}_{2     1     2     5     8     1} = 1,\\ 
\mc{A}_{2     1     2     5     8     4}
= \mc{A}_{2     2     1     4     4     2}
= \mc{A}_{2     2     1     4     4     7}
= \mc{A}_{2     2     1     4     7     5}
= \mc{A}_{2     2     1     4     7     8}
= \mc{A}_{2     2     3     3     1     6}
= \mc{A}_{2     2     3     3     1     9} = 1,\\
\mc{A}_{2     2     3     3     7     7}
= \mc{A}_{2     2     3     5     2     4}
= \mc{A}_{2     2     3     5     2     9}
= \mc{A}_{2     2     3     5     3     2}
= \mc{A}_{2     2     3     5     3     3} = 1.
\earay
\]
For $k = 6$, \reff{eq:flat_k_det} has a unique (up to scaling) nontrivial solution $x^{(0)}$.
This gives the p.o.t. $\mc{B} \coloneqq \mc{T}_6(x^{(0)})$ with observed entries:
\[
\baray{l}
\mc{B}_{1     2     3     1     3}
= \mc{B}_{1     2     3     1     4}
= \mc{B}_{1     2     3     3     6}
= \mc{B}_{2     1     1     1     4}
= \mc{B}_{2     1     1     1     8}
= \mc{B}_{2     1     1     2     2} = 1, \\
\mc{B}_{2     1     1     2     5} 
= \mc{B}_{2     1     2     4     1} 
= \mc{B}_{2     1     2     4     6} 
= \mc{B}_{2     1     2     5     5}
= \mc{B}_{2     1     2     5     8}
= \mc{B}_{2     2     1     4     4} = 1, \\
\mc{B}_{2     2     1     4     7}
= \mc{B}_{2     2     3     3     1}
= \mc{B}_{2     2     3     3     7}
= \mc{B}_{2     2     3     5     2}
= \mc{B}_{2     2     3     5     3} = 1.
\earay
\]
For $k = 5$, \reff{eq:flat_k_det} has a unique (up to scaling) nontrivial solution $x^{(1)}$.
This gives the p.o.t. $\mc{C} \coloneqq \mc{T}_5(x^{(1)})$ with observed entries:
\[
\baray{l}
\mc{C}_{1     2     3     1}
= \mc{C}_{1     2     3     3}
= \mc{C}_{2     1     1     1}
= \mc{C}_{2     1     1     2}
= \mc{C}_{2     1     2     4}
= \mc{C}_{2     1     2     5}
= \mc{C}_{2     2     1     4}
= \mc{C}_{2     2     3     3}
= \mc{C}_{2     2     3     5} =1.
\earay
\]
For $k = 4$, \reff{eq:flat_k_det} has a unique (up to scaling) nontrivial solution $x^{(2)}$.
This gives the p.o.t. $\mc{D} \coloneqq \mc{T}_4(x^{(2)})$ with observed entries:
\[
\baray{l}
\mc{D}_{1     2     3}
= \mc{D}_{2     1     1}
= \mc{D}_{2     1     2}
= \mc{D}_{2     2     1}
= \mc{D}_{2     2     3} = 1.
\earay
\]
For $k = 3$, \reff{eq:flat_k_det} has a unique (up to scaling) nontrivial solution $x^{(3)}$.
This gives the p.o.m. $\mc{E} \coloneqq \mc{T}_3(x^{(3)})$ with observed entries:
\[
\baray{l}
\mc{E}_{1     2}
= \mc{E}_{2     1}
= \mc{E}_{2     2} = 1.
\earay
\]
The solution space of \reff{eq:mx_det} for $\mc{E}$ is one-dimensional. 
So, $\mc{E}$ is rank one extractable. Therefore, $\mc{A}$ is rank one determinable.
\end{example}

\begin{defi}
For $t \in [m]$, we say that $\Omega$ is mod-$t$ full if 
\[
\{ i_t : (i_1 \ddd i_m) \in \Omega \} \,=\, \{1 \ddd n_t\}.
\]
\end{defi}

\begin{theorem}\label{thm:det-->unique_comp}
Assume $\Omega$ is mod-$t$ full for every $t \in [m]$ and 
$\mc{A}_{i_1 \cdots i_m} \ne 0$ for all $(i_1 \ddd i_m) \in \Omega$.
If $(\mc{A}, \Omega)$ is rank one determinable,
then $(\mc{A}, \Omega)$ has a unique rank one completion when it is rank one completable.
\end{theorem}

\begin{proof}
We prove by induction on $m$.
For the base step $m=2$, the conclusion is implied by Theorem~\ref{thm:connect<->rank-1}(ii).
Now assume the conclusion holds for the order $m-1$. 
Suppose $\mc{A}$ has two rank one completions:
$u_1 \otimes \cdots \otimes u_m$, $v_1 \otimes \cdots \otimes v_m$.
We need to show that each $u_k$ is a multiple of $v_k$.
Since $(\mc{A}, \Omega)$ is rank one determinable, there exists $k \in [m]$ such that $(\Flat_k(\mc{A}), \Omega^{(k)})$ is rank one extractable. 
The above implies $\Flat_k(\mc{A})$ has two rank one completions
\be\label{Uu_k=Vv_k}
\baray{l}
\Gamma_k^{-1}(u_1 \otimes \cdots \otimes u_{k-1} \otimes u_{k+1} \otimes \cdots \otimes u_m )  \cdot u_k^T, \\
\Gamma_k^{-1}(v_1 \otimes \cdots \otimes v_{k-1} \otimes v_{k+1} \otimes \cdots \otimes v_m ) \cdot v_k^T.
\earay
\ee
Observe that the two vectors
\[
\baray{l}
U \,\coloneqq\, \big( \Gamma_k^{-1}(u_1 \otimes \cdots \otimes u_{k-1} \otimes u_{k+1} \otimes \cdots \otimes u_m )  \big)_{S_1^{(k)}}, \\
V \,\coloneqq\, \big( \Gamma_k^{-1}(v_1 \otimes \cdots \otimes v_{k-1} \otimes v_{k+1} \otimes \cdots \otimes v_m )  \big)_{S_1^{(k)}}
\earay
\]
are both solutions to \reff{eq:flat_k_det}.
Since $\mc{A}_{i_1 \cdots i_m} \ne 0$ on $\Omega$, 
it is clear that $U, V$ are entirely nonzero.
Since $(\Flat_k(\mc{A}), \Omega^{(k)})$ is rank one extractable,
the solution space of  \reff{eq:flat_k_det} is one-dimensional.
So, $U = \af_k V$ for some scalar $\af_k \ne 0$.
Note that \reff{Uu_k=Vv_k} implies
\[
U_{\vec{i}} \cdot (u_k)_j \,=\, V_{\vec{i}} \cdot (v_k)_j
\quad \text{for all } \quad (\vec{i}, j) \in \Omega^{(k)},
\]
so $(u_k)_j = \af_k^{-1} (v_k)_j$ for all $j \in S_2^{(k)}$.
Since $\Omega$ is mod-$k$ full, $S_2^{(k)} = [n_k]$,
hence $u_k = \af_k^{-1} v_k$.
Let $x^* \in \re^{S_1^{(k)}}$ be the unique (up to scaling) solution to  \reff{eq:flat_k_det}, then
$
x^* = \beta_1 U = \beta_2 V
$
for nonzero scalars $\beta_1, \beta_2$.
Observe that the p.o.t. 
\[
\mc{T}_k(x^*) = \beta_1 \mc{T}_k(U) = \beta_2 \mc{T}_k(V)
\]
has two rank one completions
\[
\bt_1 u_1 \otimes \cdots \otimes u_{k-1} \otimes u_{k+1} \otimes \cdots \otimes u_m, \qquad
\bt_2 v_1 \otimes \cdots \otimes v_{k-1} \otimes v_{k+1} \otimes \cdots \otimes v_m.
\]
Note that $(\mc{T}_k(x^*), S_1^{(k)})$ is a p.o.t. of order $m-1$.
Since $S_1^{(k)}$ is mod-$t$ full for all $t \in [m] \setminus \{k\}$ 
and $\mc{T}_k(x^*)_{\vec{i}} \ne 0$ for all $\vec{i} \in S_1^{(k)}$,
by the inductive hypothesis, $\mc{T}_k(x^*)$ has a unique rank one completion, 
so there exist nonzero scalars $\af_l$ such that
\[
\bt_1 u_1 = \af_1  \bt_2v_1 ,  \qquad 
u_l = \af_l v_l \quad \text{for} \quad l \in [m] \setminus \{1, k\}.
\]
Also note $u_k = \af_k^{-1} v_k$.
This means the rank one completion for $(\mc{A}, \Omega)$ is unique.
\end{proof}

Under the assumption of Theorem~\ref{thm:det-->unique_comp},
we prove that if $(\mc{A}, \Omega)$ is rank one determinable, 
then it has a unique rank one completion.
However, the converse is not true, as shown in the following example.

\begin{example}
Consider the p.o.t. $\mc{A} \in \C^{3 \times 3 \times 4}$ with observed entries:
\[
\baray{l}
\mc{A}_{111} = \mc{A}_{221} = \mc{A}_{331} = \mc{A}_{132} 
= \mc{A}_{212} = \mc{A}_{322} = 1, \\
\mc{A}_{123} = \mc{A}_{233} = \mc{A}_{313} = \mc{A}_{124} = \mc{A}_{214} = 1.
\earay
\]
The solution space of \reff{eq:flat_k_det} is three-dimensional for $k=1,2$
and two-dimensional for $k=3$.
So, $(\mc{A}, \Omega)$ is not rank one determinable.
However, it has a unique rank one completion.
To see this, suppose (up to scaling)
\[
\mc{A} \,=\, (1,a,b) \otimes (1,c,d) \otimes (e,f,g,h)
\]
for nonzero $a,b,c,d,e,f,g,h \in \C$.
Comparing observed entries, one can see that
\[
\begin{gathered}
e=1, \quad ac=bd=df=af=bcf=cg=adg=bg=ch=ah=1, \\
df=af, \quad
cg=bg, \quad
ch=ah, \quad
ac=af, \quad
ac=ah,
\end{gathered}
\]
which implies $a=b=c=d=f=h$.
So, $bcf=adg$ forces $f=g$. 
Thus, all these parameters are equal.
Since $ac=bcf=1$, all of them must equal one.
Therefore, $(\mc{A}, \Omega)$ has the unique rank one completion
$
(1,1,1) \otimes (1,1,1) \otimes (1,1,1,1).
$
\end{example}

\section{A robust algorithm for rank one completion}\label{sc:TC_alg}

Let $(\mc{A}, \Omega)$ be a p.o.t. of order $m$.
We aim to find a rank one tensor $u_1 \otimes \cdots \otimes u_m$ 
in $\re^{n_1 \times \cdots \times n_m}$
such that
\be\label{TC_order_m}
\mc{A}_{i_1 \cdots i_m} \,=\, (u_1)_{i_1} \cdots (u_m)_{i_m}
\quad \text{for all} \quad (i_1\ddd i_m) \in \Omega.
\ee
In this section, we propose a recursive method to solve \reff{TC_order_m}.

For $k \in [m]$, 
consider the flattening matrix $\Flat_k(\mc A)$.
Note $(\Flat_k(\mc A), \Omega^{(k)})$ is a p.o.m. 
with the observation label set $\Omega^{(k)}$ given in \reff{Omg^k}.
Assume it has the rank one completion
$
\Flat_k(\mc A) = ab^T.
$
Then, the subvector ($S_1^{(k)}$ is defined in \reff{S1kS2k})
\[
(a)_{S_1^{(k)}} \,\coloneqq\, (a_{\vec{i}})_{\vec{i} \in S_1^{(k)}}
\] 
must be a solution to \reff{eq:flat_k_det} and thereby a solution to
\be\label{eq:Bkx=0}
B^{(k)}x \,=\, 0,
\ee
where $B^{(k)}$ is given in \reff{Bx=0}.
If $(\Flat_k(\mc A), \Omega^{(k)})$ is rank one extractable, 
then $\text{null}(B^{(k)})$ is one-dimensional.
That is, \reff{eq:Bkx=0} has a unique (up to scaling) nontrivial solution, say, $x^* \in \re^{S_1^{(k)}}$.
Then, the vector $b$ can be determined from
\[
\big( \Flat_k(\mc{A}) \big)_{\vec{i} j} \,=\, a_{\vec{i}} b_j \quad \text{for all} \quad (\vec{i}, j) \in \Omega^{(k)}.
\]
As in \reff{up_k}, the vector $x^*$ can be reshaped to the $(m-1)$st order p.o.t.
\[
\mc{T}_k(x^*) \,\in\, \re^{\mc{P}_k}
\]
with the observation label set $S_1^{(k)}$.
We can repeat the above procedure for the p.o.t. $(\mc{T}_k(x^*), S_1^{(k)})$.
This gives a recursive method to solve \reff{TC_order_m}, 
in which the tensor order decreases by one after each recursion.
This process terminates when we reach the matrix case (i.e., the order equals two).

We illustrate the above recursive procedure for $m=4$.
Suppose $(\Flat_4(\mc A), \Omega^{(4)})$ is rank one extractable.
Let $x^{(0)} \in \re^{S_1^{(4)}}$ be a nontrivial solution to \reff{eq:flat_k_det}.
Then we determine $u_4 \in \re^{n_4}$ such that 
$x^{(0)} u_4^T$ gives a rank one matrix completion for $\Flat_4(\mc A)$.
Next, reshape $x^{(0)}$ to obtain the third order p.o.t.
\[
\mc A_1 \coloneqq \mc{T}_4(x^{(0)}) \in \re^{n_1\times n_2\times n_3}
\]
with the observation label set $\Omega_1 \coloneqq S_1^{(4)}$.
Assume the p.o.m. $\Flat_3(\mc A_1)$ is rank one extractable.
Let $x^{(1)}$ be a nontrivial solution to \reff{eq:flat_k_det} for the p.o.t. $\mc{A}_1$.
Then, determine $u_3 \in \re^{n_3}$ such that $x^{(1)} u_3^T$ gives a rank one matrix completion for
$\Flat_3(\mc{A}_1)$.
Next, reshape $x^{(1)}$ to obtain the p.o.m.
\[
\mc{A}_2 \coloneqq \mc{T}_3(x^{(1)}) \in \re^{n_1\times n_2}
\]
with the observation label set $\Omega_2$.
Assume $(\mc{A}_2, \Omega_2)$ is rank one extractable. 
Let $x^{(2)}$ be a nontrivial solution to \reff{eq:flat_k_det} for the p.o.m. $\mc{A}_2$.
Then, determine $u_2 \in \re^{n_2}$ such that $x^{(2)}u_2^T$ gives a rank one matrix completion for $\mc{A}_2$.
Finally, let $u_1 = x^{(2)}$.
Then, under some general assumptions (see Theorem~\ref{thm:alg_unique}),
we get the rank one tensor completion 
$u_1 \otimes u_2 \otimes u_3 \otimes u_4$ for $(\mc{A}, \Omega)$.

In computational practice, how do we choose $k$ such that $(\Flat_k(\mc A),\Omega^{(k)})$ is rank one extractable?
This requires the nullspace of $B^{(k)}$ to be one-dimensional.
However, when observed tensor entries have noise, 
$B^{(k)}$ may have only trivial nullspace.
For such a case, we can estimate the nullspace as follows
(let $B^{(k)}$ be $\ell$-by-$n$):
\bit
\item If $\ell \ge n$, then $B^{(k)}$ has singular values 
\be\label{sigs}
\sig_1 \ge \cdots \ge \sig_n \ge 0.
\ee
We compute the right singular vector $v_n$ of $B^{(k)}$ corresponding to $\sig_n$.
Then, $\Span\{v_n\}$ is an approximation for the nullspace of $B^{(k)}$.

\item If $\ell < n$, then $B^{(k)}$ must have a nontrivial nullspace, 
and we compute $\text{null}(B^{(k)})$ directly.
\eit

Summarizing the above, we now present a recursive algorithm that solves \reff{TC_order_m}.

\begin{alg}\label{alg:recursive}
Let $(\mc{A}, \Omega)$ be a p.o.t. of order $m \ge 2$.
Do:

\bnum

\item[Step~1] For each $k \in [m]$,
formulate $B^{(k)}$ as in \reff{Bx=0}.
Compute the minimum singular value of $B^{(k)}$, denoted by $\sig_{\min}(B^{(k)})$.
Choose $k$ such that
\be\label{sig(Bk)}
\sig_{\min}(B^{(k)}) \,=\, \min_{i \in [m]} \sig_{\min}(B^{(i)}).
\ee
If the above $k$ is not unique, choose the $k$ for which the gap 
$\sig_{n-1} - \sig_n$ is maximum, where $\sig_n, \sig_{n-1} $ are as in \reff{sigs}.

\item[Step~2] Compute the vector $x^* \in \re^{S_1^{(k)}}$ as follows
(let $B^{(k)}$ be $\ell$-by-$n$):
\bit
\item[i)] If $\ell \ge n$, compute the right singular vector
$v_n$ of $B^{(k)}$ corresponding to $\sig_{\min}(B^{(k)})$. 
Let $x^* = v_n$.

\item[ii)] If $\ell < n$, directly solve $B^{(k)}x = 0$ for a nontrivial solution $x^*$.
Normalize $x^*$ so that $\|x^*\|=1$. 
\eit

\item[Step~3] 
As in \reff{up_k}, reshape $x^*$ to obtain the p.o.t. $(\mc{T}_k(x^*), S_1^{(k)})$ of order $m-1$.

\item[Step~4] 
If $m>2$, compute a rank one tensor completion for the p.o.t. $\mc{T}_k(x^*)$ as
\be\label{u_1...u_m - u_k}
u_1 \otimes \cdots \otimes u_{k-1} \otimes u_{k+1} \otimes \cdots \otimes u_m.
\ee
If $m=2$, then $\mc{T}_k(x^*)$ has order one (i.e., is a vector), 
and $\mc{P}_k$ is one-dimensional, say, $\mc{P}_k = [n_t]$ for some $t$.
If $S_1^{(k)} \ne [n_t]$,
we extend $\mc{T}_k(x^*)$ to a full dimensional vector in $\re^{\mc{P}_k}$.
Otherwise, $\mc{T}_k(x^*)$ is automatically a rank one completion for itself.

\item[Step~5] Obtain $u_k \in \re^{n_k}$ by solving the linear least squares problem
\be\label{LS_induct}
\min_{u} \quad \| \mc{A} - u_1 \otimes \cdots \otimes u_{k-1} \otimes u \otimes u_{k+1} \otimes \cdots \otimes u_m \|_{\Omega}.
\ee

\item[Step~6] Output the rank one decomposing tuple $(u_1 \ddd u_m)$.

\enum
\end{alg}

\begin{remark}
In Step~4, Algorithm~\ref{alg:recursive} can be applied recursively to the $(m-1)$st order p.o.t. $\mc{T}_k(x^*)$ to obtain the rank one completion \reff{u_1...u_m - u_k}.
Each recursive step reduces the tensor order by one, 
and it terminates when we reach the matrix case.
\end{remark}

We apply Algorithm~\ref{alg:recursive} to $(\mc{A},\Omega)$ recursively as follows.
Set $(\mc{A}_0,\Omega_0)\coloneqq(\mc A,\Omega)$.
Let $k_0 \in [m]$ be the index chosen in Step~1, 
and let $x^{(0)}$ be the vector $x^*$ obtained from $B_0^{(k_0)}$ in Step~2.
Step~3 produces the p.o.t. $(\mc{T}_{k_0}(x^{(0)}), S_1^{(k_0)})$.
Define 
\[
\mc{A}_1 \coloneqq \mc{T}_{k_0}(x^{(0)}),
\quad \Omega_1 \coloneqq S_1^{(k_0)}.
\]
We apply Algorithm~\ref{alg:recursive} again to the $(m-1)$st order p.o.t. $(\mc{A}_1, \Omega_1)$.
Let $k_1 \in [m] \setminus \{k_0\}$ be the index chosen in Step~1,
and let $x^{(1)}$ be the vector $x^*$ obtained from $B_1^{(k_1)}$ in Step~2.
Step~3 produces the p.o.t. $\mc{T}_{k_1}(x^{(1)})$.
Define 
\[
\mc{A}_2 \,\coloneqq\, \mc{T}_{k_1}(x^{(1)}),
\]
and let $\Omega_2$ denote its observation label set.
Again, we apply Algorithm~\ref{alg:recursive} to the $(m-2)$nd order p.o.t. $(\mc{A}_2, \Omega_2)$.
Let $k_2 \in [m] \setminus \{k_0, k_1 \}$ be the index chosen in Step~1,
and let $x^{(2)}$ be the vector $x^*$ obtained from $B_2^{(k_2)}$ in Step~2.
Step~3 produces the p.o.t. $\mc{A}_3 \coloneqq \mc{T}_{k_2}(x^{(2)})$.
This recursive procedure gives the chain of p.o.t.'s
\be\label{chain}
(\mc{A}_0, \Omega_0) 
\overset{k_0}{\longrightarrow} (\mc{A}_1, \Omega_1) 
\overset{k_1}{\longrightarrow} (\mc{A}_2, \Omega_2)
\overset{k_2}{\longrightarrow}
\cdots 
\overset{k_{m-2}}{\longrightarrow} (\mc{A}_{m-1}, \Omega_{m-1}).
\ee
Note that each p.o.t. $\mc A_j$ has order $m-j$ and the index
\be\label{k_j}
k_j \in [m] \setminus \{ k_0 \ddd k_{j-1} \}.
\ee

\begin{example}
Consider the p.o.t. $\mc{A} \in \re^{2 \times 2 \times 2 \times 2 \times 2}$ with observed entries:
\[
\mc{A}_{11111} = \mc{A}_{11112} = \mc{A}_{11222  } = \mc{A}_{12111} 
= \mc{A}_{21211} = \mc{A}_{21222} =  \mc{A}_{22121} = 1.
\]
When Algorithm~\ref{alg:recursive} is applied recursively,
the chain of p.o.t.'s as in \reff{chain} is
\[
\baray{l}
(\mc{A}_1)_{1     1     1     1} 
= (\mc{A}_1)_{1     1     2     2}
= (\mc{A}_1)_{1     2     1     1}
= (\mc{A}_1)_{2     1     2     1}
= (\mc{A}_1)_{2     1     2     2}
= (\mc{A}_1)_{2     2     1     2} = \frac{1}{\sqrt{6}}, \\
(\mc{A}_2)_{1     1     1}
= (\mc{A}_2)_{1     1     2}
= (\mc{A}_2)_{1     2     1}
= (\mc{A}_2)_{2     1     2}
= (\mc{A}_2)_{2     2     1} = \frac{1}{\sqrt{5}}, \\
(\mc{A}_3)_{1     1}
=(\mc{A}_3)_{1     2}
= (\mc{A}_3)_{2     1}
= (\mc{A}_3)_{2     2} = \frac{1}{2},
\qquad
(\mc{A}_4)_1 = (\mc{A}_4)_2 = \frac{1}{\sqrt{2}}.
\earay
\]
The obtained rank one completion is
$\mc{A} = u_1 \otimes u_2 \otimes u_3 \otimes u_4 \otimes u_5$ with
\[
\baray{c}
u_1 = u_2 = \big( \frac{1}{\sqrt{2}},
\frac{1}{\sqrt{2}} \big),
\,\, u_3 = \big( \frac{2}{\sqrt{5}},
\frac{2}{\sqrt{5}} \big),
\,\, u_4 = \big( \sqrt{\frac{5}{6}}, \sqrt{\frac{5}{6}} \big),
\,\, u_5 = (\sqrt{6}, \sqrt{6}).
\earay
\]
\end{example}

\begin{theorem}\label{thm:alg_unique}
Assume $\Omega$ is mod-$t$ full for every $t \in [m]$ and
$\mc{A}_{i_1 \cdots i_m} \ne 0$ for all $(i_1 \ddd i_m) \in \Omega$.
Suppose $(\mc{A}, \Omega)$ is rank one completable and rank one determinable.
Let $(\mc{A}_j,\Omega_j, k_j)$ be as in \reff{chain}.
Assume $(\Flat_{k_j}(\mc{A}_j), \Omega_j^{(k_j)})$ is rank one extractable for all $0 \le j \le m-2$.
Then, the tuple $(u_1 \ddd u_m)$ produced by Algorithm~\ref{alg:recursive} 
gives the unique rank one completion
$
\mc{A} = u_1 \otimes \cdots \otimes u_m.
$
\end{theorem}

\begin{proof}
We prove by induction on $m$ that Algorithm~\ref{alg:recursive} gives a rank one completion.
The uniqueness is implied by Theorem~\ref{thm:det-->unique_comp}.

For the base step $m=2$, $A \coloneqq \mc{A}$ is a matrix.
Without loss of generality, we can assume $k_0=2$ in Step~1.
Then the matrix $B^{(2)}$ as in \reff{Bx=0} is the same as the matrix $B$ in \reff{Bx=0_mx}.
Since $(A, \Omega)$ is rank one extractable, $\text{null}(B)$ is one-dimensional.
Then, for the vector $x^*$ produced in Step~2, 
we have $\text{null}(B) = \Span\{x^*\}$,
and $x^*$ is a solution to \reff{eq:mx_det}.
By Theorem~\ref{thm:connect<->rank-1}(ii), 
the graph $G(V_1, V_2, \Omega)$ is connected.
So, $x^*$ must be entirely nonzero by Theorem~\ref{thm:det_comp}(iv).
Note that $\mc{T}_{k_0}(x^*)$ is a vector, so 
Step~4 produces $u_1 = x^*$, as $\Omega$ is row-full.
By Theorem~\ref{thm:det_comp}(ii), 
there must exist $v_2$ such that $A = u_1v_2^T$.
This means $v_2$ is a minimizer of 
\be\label{LS_mx}
\min_u \quad \| A - u_1 \otimes u \|_{\Omega},
\ee
and the minimum value is zero.
Note that $u_2$ produced in Step~5 is also a solution to \reff{LS_mx}.
Since $\Omega$ is column-full,
$u_2 \in \re^{n_2}$ can be uniquely determined by
\[
(u_2)_j \,=\, A_{ij} / (u_1)_i \quad \text{for all } (i,j) \in \Omega.
\]
This means $u_2 = v_2$.
Hence, $A = u_1u_2^T$ is a rank one completion for $(A, \Omega)$.

For the inductive step, assume the conclusion holds for the order $m-1$. 
Since $(\mc{A}_0, \Omega_0)$ is rank one completable, there exists a rank one completion  
\be\label{A=tu}
\mc{A}_0 \,=\, v_1 \otimes \cdots \otimes v_m.
\ee
Recall that $\Omega_1 = S_1^{(k_0)}$ as in \reff{S1kS2k}.
Then, the vector
\be\label{xtilde}
\hat{x} \,\coloneqq\, \big( \Gamma_{k_0}^{-1}( v_1 \otimes \cdots \otimes v_{k_0-1} \otimes v_{k_0+1} \otimes \cdots \otimes v_m ) \big)_{\Omega_1} 
\ee
is a solution to \reff{Bx=0}, i.e., $B^{(k_0)} \hat{x} = 0$.
Since $\mc{A}_{i_1 \cdots i_m} \ne 0$ on $\Omega$, 
it is clear that $\hat{x}$ is entirely nonzero.
Since $(\Flat_{k_0}(\mc{A}_0),\Omega_0^{(k_0)})$ is rank one extractable, 
$\text{null}(B^{(k_0)})$ is one-dimensional.
Hence, the vector $x^{(0)} = x^*$ obtained in Step~2 satisfies
$x^{(0)} = \af_{k_0} \hat{x}$ for some scalar $\af_{k_0} \ne 0$.
Step~3 reshapes $x^{(0)}$ to the $(m-1)$st order p.o.t.
\be\label{t(x0) = afT(xtild)}
\mc{A}_1 \coloneqq \mc{T}_{k_0}(x^{(0)}) 
= \af_{k_0} \mc{T}_{k_0}(\hat{x}).
\ee
By \reff{xtilde} and \reff{t(x0) = afT(xtild)}, $(\mc{A}_1, \Omega_1)$ has the rank one completion:
\[
\af_{k_0} v_1 \otimes \cdots \otimes v_{k_0-1} \otimes v_{k_0+1} \otimes \cdots \otimes v_m.
\]
By the given assumption, $(\mc{A}_1, \Omega_1)$ is rank one determinable,
$\Omega_1$ is mod-$t$ full for all $t \in [m] \setminus \{k_0\}$,
and all observed entries of $\mc{A}_1$ are nonzero.
By the inductive hypothesis, the algorithm gives the unique rank one completion
\[
u_1 \otimes \cdots \otimes u_{k_0-1} \otimes u_{k_0+1} \otimes \cdots \otimes u_m
\]
for $(\mc{A}_1, \Omega_1)$.
Hence, there exist scalars $\af_l \ne 0$ such that
\be\label{u=aftu}
u_l = \af_l v_l \quad \text{for all} \quad l \in [m] \setminus \{k_0\}. 
\ee
By \reff{A=tu} and \reff{u=aftu},
\begin{align*}
\mc{A} 
&= v_1 \otimes \cdots \otimes v_{k_0-1} \otimes v_{k_0} \otimes v_{k_0+1} \otimes \cdots \otimes v_m \\
&= u_1 \otimes \cdots \otimes u_{k_0-1} \otimes u^* \otimes u_{k_0+1} \otimes \cdots \otimes u_m,
\end{align*}
where
$
u^* = \af_{k_0} v_{k_0} / (\af_1 \cdots \af_m).
$
So, $u^*$ is a minimizer of \reff{LS_induct}, and the minimum value is zero.
Note that $u_{k_0}$ produced in Step~5 is also a solution to \reff{LS_induct}.
Since $\Omega$ is mod-$k_0$ full, 
the solution to \reff{LS_induct} is unique, i.e., $u_{k_0} = u^*$.
Hence, $u_1 \otimes \cdots \otimes u_m$ is a rank one completion for $(\mc{A}, \Omega)$.
\end{proof}

\section{Perturbation analysis}\label{sc:perturb}

Let $(\mc{A}, \Omega)$ be a p.o.t. of order $m$.
In applications, there is often noise in the observed tensor entries.
We consider the perturbation
\[
\widetilde{\mc{A}} \,=\, \mc{A}  + \dt \mc{A},
\]
where $\dt \mc{A}$ denotes the noise tensor.
Suppose $(\mc{A}, \Omega)$ is rank one completable, 
and $u_1 \otimes \cdots \otimes u_m$ is a rank one completion given by Algorithm~\ref{alg:recursive}.
Let $(\tu_1 \ddd \tu_m)$ be the output of Algorithm~\ref{alg:recursive} applied to $(\widetilde{\mc{A}}, \Omega)$.
When $\|\dt \mc{A}\|_{\Omega}$ is sufficiently small, we show that
$u_1 \otimes \cdots \otimes u_m$ is close to $\tu_1 \otimes \cdots \otimes \tu_m$.

Recall that Algorithm~\ref{alg:recursive} produces the chain of $(\mc{A}_j, \Omega_j)$ as in \reff{chain}.
Similarly, we apply Algorithm~\ref{alg:recursive} to the perturbation 
$(\widetilde{\mc{A}}_0,\tomg_0)\coloneqq(\widetilde{\mc{A}},\tomg)$.
We assume the observation label set remains unchanged, i.e., $\Omega = \tomg$.
Let $\tk_0 \in [m]$ be the index chosen in Step~1, 
and let $\tx^{(0)}$ be the vector $x^*$ obtained from $\widetilde{B}_0^{(\tk_0)}$ in Step~2.
Up to a sign, we can generally assume $(x^{(0)})^T\tx^{(0)} \ge 0$.
Step~3 produces the $(m-1)$st order p.o.t. $\widetilde{\mc{A}}_1 \coloneqq \mc{T}_{\tk_0}(\tx^{(0)})$, and let $\tomg_1 \coloneqq S_1^{(\tk_0)}$ be its observation label set.
Again, apply Algorithm~\ref{alg:recursive} to $(\widetilde{\mc{A}}_1, \tomg_1)$.
Let $\tk_1 \in [m] \setminus \{\tk_0\}$ be the index chosen in Step~1,
and let $\tx^{(1)}$ be the vector obtained from $\widetilde{B}_1^{(\tk_1)}$ in Step~2.
Again, up to a sign, we can assume $(x^{(1)})^T\tx^{(1)} \ge 0$.
Step~3 produces the p.o.t. $\widetilde{\mc{A}}_2 \coloneqq \mc{T}_{\tk_1}(\tx^{(1)})$,
and let $\tomg_2$ denote its label set.
Repeating the above gives the chain of p.o.t.'s
\be\label{chain_pert}
(\widetilde{\mc{A}}_0, \tomg_0) 
\overset{\tk_0}{\longrightarrow} (\widetilde{\mc{A}}_1, \tomg_1) 
\overset{\tk_1}{\longrightarrow} (\widetilde{\mc{A}}_2, \tomg_2)
\overset{\tk_2}{\longrightarrow}
\cdots 
\overset{\tk_{m-2}}{\longrightarrow} (\widetilde{\mc{A}}_{m-1},\tomg_{m-1}).
\ee
Note that each $\widetilde{\mc{A}}_j$ has order $m-j$ and 
$
\tk_j \in [m] \setminus \{ \tk_0 \ddd \tk_{j-1} \}.
$

\begin{theorem}\label{thm:tensor_perturb}
Let $(\mc{A}, \Omega)$ be an $m$th order p.o.t. such that
$
\mc{A} = u_1 \otimes \cdots \otimes u_m
$
is the rank one completion given by Algorithm~\ref{alg:recursive}.
Assume $(\mc{A}, \Omega)$ is rank one determinable,
$\Omega$ is mod-$t$ full for every $t \in [m]$, and 
$\mc{A}_{i_1 \cdots i_m} \ne 0$ for all $(i_1 \ddd i_m) \in \Omega$.
Consider a perturbation $\widetilde{\mc{A}} \coloneqq \mc{A} + \dt \mc{A}$ on $\Omega = \tomg$.
Let $(\mc{A}_j,\Omega_j, k_j)$ and $(\widetilde{\mc{A}}_j,\tomg_j, \tk_j)$
be as in \reff{chain} and \reff{chain_pert} respectively.
For $j = 0,1 \ddd m-2$, suppose
\bnum
\item[(a)] each $(\Flat_{k_j}(\mc{A}_j), \Omega_j^{(k_j)})$ is rank one extractable;

\item[(b)] there exists some scalar $\mu_j > 0$ such that
\[
\min_{i \in [m] \setminus \{ k_0 \ddd k_{j} \}} \sig_{\min}(B_j^{(i)}) \,\ge\, \sig_{\min}(B_j^{(k_j)}) + \mu_j.
\]
\enum
Let $(\tu_1 \ddd \tu_m)$ be the output of Algorithm~\ref{alg:recursive} 
applied to $(\widetilde{\mc{A}}, \tomg)$.
When $\|\dt \mc{A}\|_{\Omega}$ is sufficiently small, we have:
\bnum
\item[(i)] For $j = 0,1 \ddd m-2$, it holds that
$\tk_j = k_j$, $\tomg_j = \Omega_j$, and there exists a constant $\rho_j > 0$ such that
\be\label{tA-A<=dA}
\| \widetilde{\mc{A}}_{j+1} - \mc{A}_{j+1} \|_{\Omega_{j+1}} \,\le\, \rho_j \|\dt \mc{A}\|_{\Omega}.
\ee

\item[(ii)] There exist constants $C_1>0$, $C_2 > 0$, depending on $(u_1 \ddd u_m)$, such that
\begin{gather}
\label{tu_i-u_i<=dA}
\| \tu_t - u_t \| \,\le\, C_1 \|\dt \mc{A}\|_{\Omega} \quad \text{for} \quad t = 1 \ddd m, \\
\label{tu-u<=dA}
\| \tu_1 \otimes \cdots \otimes \tu_m - u_1 \otimes \cdots \otimes u_m \| \,\le\, C_2 \|\dt \mc{A}\|_{\Omega}.
\end{gather}
\enum
\end{theorem}

\begin{proof}
(i) For $j=0$,
it is clear that $\tomg_0 = \tomg = \Omega = \Omega_0$.
For $i \in [m]$, 
let $B_0^{(i)}$ and $\widetilde{B}_0^{(i)}$ be the matrices constructed from 
$(\mc{A}_0, \Omega_0)$ and $(\widetilde{\mc{A}}_0, \tomg_0)$ as in \reff{Bx=0}.
Denote
\[
\dt \mc{A}_0 \coloneqq \widetilde{\mc{A}}_0 - \mc{A}_0, 
\quad \dt B_0^{(i)} \coloneqq \widetilde{B}_0^{(i)} - B_0^{(i)}.
\]
The nonzero entries of $B_0^{(i)}$ are observed tensor entries (up to a sign) of $\mc{A}_0$,
so there exists a constant $D_0 > 0$, depending on $\Omega_0$, such that
\be\label{dB0<=dA0}
\|\dt B_0^{(i)}\| \,\le\, D_0 \|\dt \mc{A}_0 \|_{\Omega_0}.
\ee
Since the above holds for all $i \in [m]$, we have
\be\label{dB<=mu/3_0}
\max_{i \in [m]} \|\dt B_0^{(i)}\| \,\le\, \frac{\mu_0}{3},
\ee 
when $\|\dt \mc{A}\|_{\Omega}$ is sufficiently small.
By Weyl's theorem \cite[Corollary~5.1]{demmel97}, for all $i \in [m]$,
\be\label{eq:weyl_0}
| \sig_{\min}(\widetilde{B}_0^{(i)}) - \sig_{\min}(B_0^{(i)}) | \,\le\, \|\dt B_0^{(i)}\|.
\ee
Let $k_0 \in [m]$ be the index chosen in Step~1 of Algorithm~\ref{alg:recursive} applied to $(\mc{A}_0, \Omega_0)$.
Then assumption~(b) gives
\[
\min_{i \in [m] \setminus \{k_0\}} \sig_{\min}(B_0^{(i)}) \,\ge\, \sig_{\min}(B_0^{(k_0)}) + \mu_0.
\]
Then, by \reff{eq:weyl_0}, for all $i \ne k_0$,
\be\label{sigBji>=?_0}
\baray{rcl}
\sig_{\min}(\widetilde{B}_0^{(i)}) &\ge& \sig_{\min}(B_0^{(i)}) - \|\dt B_0^{(i)}\| \\
&\ge& \sig_{\min}(B_0^{(k_0)}) + \mu_0 - \|\dt B_0^{(i)}\| \\ 
&\ge& \sig_{\min}(\widetilde{B}_0^{(k_0)}) - \|\dt B_0^{(k_0)}\| + \mu_0 - \|\dt B_0^{(i)}\|.
\earay
\ee
For all $i \ne k_0$, when $\|\dt \mc{A}\|_{\Omega}$ is sufficiently small, \reff{dB<=mu/3_0} implies
\begin{eqnarray*}
\sig_{\min}(\widetilde{B}_0^{(i)}) &\,\ge\,& \sig_{\min}(\widetilde{B}_0^{(k_0)}) + \frac{\mu_0}{3}, \\
\sig_{\min}(\widetilde{B}_0^{(k_0)}) &\,<\,& \min\limits_{i \in [m] \setminus \{k_0\}} \sig_{\min}(\widetilde{B}_0^{(i)}).
\end{eqnarray*}
This forces $\tk_0 = k_0$ and hence $\tomg_1 = \Omega_1$.

By assumption~(a), $\text{null}(B_0^{(k_0)})$ is one-dimensional.
Recall that $x^{(0)}$ (resp., $\tx^{(0)}$) is the unit length vector computed from $B_0^{(k_0)}$
(resp., $\widetilde{B}_0^{(\tk_0)} = \widetilde{B}_0^{(k_0)}$) in Step~2 of Algorithm~\ref{alg:recursive}.
Up to a sign, we can assume $(x^{(0)})^T\tx^{(0)} \ge 0$.
By Theorem~\ref{thm:mx_null_perturb}, there exists $C_0 > 0$ such that
\[
\|\tx^{(0)} - x^{(0)}\| \,\le\, C_0 \|\dt B_0^{(k_0)}\|.
\]
By inequality~\reff{dB0<=dA0}, we get
\be\label{dx0<=dA0}
\|\tx^{(0)} - x^{(0)}\| \,\le\, C_0D_0 \|\dt \mc{A}_0 \|_{\Omega_0}.
\ee
Note that
$\mc{A}_1 = \mc{T}_{k_0}(x^{(0)})$,
$\widetilde{\mc{A}}_1 = \mc{T}_{k_0}(\tx^{(0)})$,
and $\tomg_1 = \Omega_1$.
Observe that $\widetilde{\mc{A}}_1 - \mc{A}_1$ is a reshaping of $\tx^{(0)} - x^{(0)}$, so
\[
\| \widetilde{\mc{A}}_1 - \mc{A}_1 \|_{\Omega_1} \,=\, \|\tx^{(0)} - x^{(0)}\|.
\]
By \reff{dx0<=dA0}, the above implies that when $\|\dt \mc{A}\|_{\Omega}$ is sufficiently small,
\be\label{dA1<=rho0dA0}
\| \widetilde{\mc{A}}_1 - \mc{A}_1 \|_{\Omega_1} \,\le\, \rho_0 \|\dt \mc{A} \|_{\Omega}
\quad \text{for} \quad \rho_0 \coloneqq C_0D_0.
\ee

Now we view $(\mc{A}_1, \Omega_1)$ as $(\mc{A}_0, \Omega_0)$ and repeat the above argument.
We can show that $\tk_1 = k_1$ (hence $\tomg_2 = \Omega_2$) and \reff{tA-A<=dA} holds for $j=1$.
Similarly, one can show that the conclusion holds for $j = 2 \ddd m-2$.

\noindent (ii) We prove by induction on $m$.
For the base step $m=2$, we apply Algorithm~\ref{alg:recursive} to 
$(\mc{A}_0, \Omega_0) = (\mc{A}, \Omega)$.
Without loss of generality, we assume $k_0 = 2$ in Step~1 of Algorithm~\ref{alg:recursive}.
The matrix $B_0^{(2)}$ formulated as in \reff{Bx=0} is the same as the matrix $B$ in \reff{Bx=0_mx}.
By assumption~(a), $\text{null}(B)$ is one-dimensional.
Recall that $x^{(0)}$ is the unit length vector computed from $B$ in Step~2 of Algorithm~\ref{alg:recursive}.
Since $\Omega$ is row-full, $u_1 = x^{(0)}$.
Similarly, $\tx^{(0)}$ is the unit length vector computed from $\widetilde{B}$ in Step~2 when Algorithm~\ref{alg:recursive} is applied to $(\widetilde{\mc{A}}, \tomg)$ where $\tomg = \Omega$.
Also since $\Omega$ is row-full, $\tu_1 = \tx^{(0)}$.
Let $\dt B \coloneqq \widetilde{B} - B$.
By \reff{dx0<=dA0}, 
\be\label{u1bound}
\|\tu_1 - u_1\| = \|\tx^{(0)} - x^{(0)}\| \le C_0 D_0 \|\dt \mc{A}\|_{\Omega}.
\ee
In Step~5, Algorithm~\ref{alg:recursive} computes $u_2$ (resp., $\tu_2$) 
from $(\mc{A}, \Omega)$ (resp., $(\widetilde{\mc{A}}, \Omega)$) by solving
a linear least squares problem.
Fix an arbitrary $j \in [n_2]$, 
we write 
\[
\Theta_j \,=\, \{ (i_1, j), (i_2, j) \ddd (i_{m_j}, j) \},
\]
where $\Theta_j$ is defined as in \reff{theta_j}.
Denote 
\[
\baray{lcl}
a \,\coloneqq\, \big[ 
\mc{A}_{i_1 j} \,\, \cdots \,\, \mc{A}_{i_{m_j} j}
\big]^T, &&
\ta \,\coloneqq\, \big[ 
\widetilde{\mc{A}}_{i_1 j} \,\,  \cdots \,\,  \widetilde{\mc{A}}_{i_{m_j} j}
\big]^T, \\
w \,\coloneqq\, \big[ 
(u_1)_{i_1} \,\,  \cdots \,\,  (u_1)_{i_{m_j}}
\big]^T, &&
\tw \,\coloneqq\, \big[ 
(\tu_1)_{i_1} \,\,  \cdots \,\,  (\tu_1)_{i_{m_j}}
\big]^T.
\earay
\]
Then the entries $(u_2)_j$ and $(\tu_2)_j$ are, respectively, the minimizers for
\be\label{ls_w}
\min_{\theta \in \re^1} \, \|a - \theta w\|, \qquad
\min_{\tilde{\theta} \in \re^1} \, \|\ta - \tilde{\theta} \tw\|.
\ee
Since $\mc{A}_{ij} \ne 0$ for all $(i,j) \in \Omega$ and $\mc{A} = u_1u_2^T$,
we have $w \ne 0$ and $(u_2)_j \ne 0$.
Note that $a = (u_2)_j w$, so
$
|w^Ta| = |(u_2)_j| \cdot \|w\|^2 > 0.
$
Let $\dt a \coloneqq \ta - a$ and $\dt w \coloneqq \tw - w$.
Since $\dt a$ is a sub-array of $\dt\mc{A}$, we have
\be\label{da<=dA}
\|\dt a\| \,\le\, \|\dt \mc{A}\|_{\Omega}.
\ee
Since $\dt w$ is a subvector of $\tu_1 - u_1$, \reff{u1bound} implies
\be\label{dw<=dA}
\|\dt w\| \le \|\tu_1 - u_1\| \le C_0 D_0 \|\dt \mc{A}\|_{\Omega}.
\ee
The least squares solutions to \reff{ls_w} can be expressed as
\[
(u_2)_j = w^Ta \big/ w^Tw,
\quad (\tu_2)_j = \tw^T\ta \big/ \tw^T\tw.
\]
Observe that
\begin{align*}
(\tu_2)_j - (u_2)_j &= \frac{(w+\dt w)^T(a + \dt a)}{(w + \dt w)^T(w + \dt w)} - \frac{w^Ta}{w^Tw}
= (u_2)_j  \cdot \frac{\tau - \gm}{1 + \gm},
\end{align*}
where
\[
\tau \coloneqq \frac{w^T(\dt a) + (\dt w)^Ta + (\dt w)^T(\dt a)}{w^Ta},
\qquad \gm \coloneqq \frac{2w^T(\dt w) + (\dt w)^T (\dt w)}{w^Tw}.
\]
By the inequality $w^T(\dt w) \ge - \|w\| \|\dt w\|$, we get
\[
\gm \ge \frac{ -2 \|w\| \|\dt w\| + \|\dt w\|^2 }{\|w\|^2}
= -2\frac{\|\dt w\|}{\|w\|} + \Big(\frac{\|\dt w\|}{\|w\|}\Big)^2.
\]
When $\|\dt \mc{A}\|_{\Omega}$ is sufficiently small, 
we have $1+\gm \ge \frac{1}{2}$ by \reff{dw<=dA}, so
\[
\frac{|(\tu_2)_j - (u_2)_j|}{|(u_2)_j|} = \Big| \frac{\tau - \gm}{1+\gm} \Big| 
\le 2|\tau - \gm| \le 2(|\tau| + |\gm|).
\]
Note that 
\[
|\tau| \le \frac{\|w\|\|\dt a\| + \|\dt w\| \|a\| + \|\dt w\| \|\dt a\|}{|w^Ta|},
\qquad |\gm| \le \frac{2 \|w\| \|\dt w\| + \|\dt w\|^2}{\|w\|^2}.
\]
By \reff{da<=dA} and \reff{dw<=dA}, there exists a constant $L_j > 0$ such that
\[
|(\tu_2)_j - (u_2)_j| \,\le\, L_j \|\dt \mc{A}\|_{\Omega}.
\]
Let $K \coloneqq \sqrt{n_2} \max_j L_j$, then
\be\label{u2bound}
\|\tu_2 - u_2\| = \Big( \sum_{j = 1}^{n_2} |(\tu_2)_j - (u_2)_j|^2 \Big)^{1/2} \le K \|\dt \mc{A}\|_{\Omega}.
\ee
Hence, \reff{tu_i-u_i<=dA} holds with $C_1 \coloneqq \max\{ C_0D_0, K \}$.
If $\|\dt \mc{A}\|_{\Omega} \le \frac{1}{K}$, then \reff{u2bound} implies
\[
\|\tu_2\| \le \| \tu_2 - u_2 \| + \|u_2\| \le K \|\dt \mc{A}\|_{\Omega} + \|u_2\| \le 1 +  \|u_2\|.
\]
Combining \reff{u1bound} and \reff{u2bound}, we get
\be\label{base}
\baray{rcl}
\|\tu_1\otimes \tu_2 - u_1 \otimes u_2\|
&=& \|\tu_1 \otimes \tu_2 - u_1 \otimes \tu_2 + u_1 \otimes \tu_2 -u_1\otimes u_2\| \\
&\le& \|\tu_1-u_1\|\,\|\tu_2\|+\|u_1\|\,\|\tu_2-u_2\| \\
&\le& C_2\|\dt\mc{A}\|_\Omega
\earay
\ee
for the constant $C_2 \coloneqq C_0D_0(1+\|u_2\|)+ K\|u_1\|$, depending on $(u_1, u_2)$.

For the inductive step, assume the conclusion holds for the order $m-1$. 
We apply Algorithm~\ref{alg:recursive} to the $(m-1)$st order p.o.t.'s 
$(\mc{A}_1, \Omega_1)$ and $(\widetilde{\mc{A}}_1, \tomg_1)$.
By item~(i),
$\tomg_1 = \Omega_1$, $\widetilde{\mc{A}}_1$ and $\mc{A}_1$ have the same dimension.
Let
$\dt\mc{A}_1 \coloneqq \widetilde{\mc{A}}_1 - \mc{A}_1$.
Without loss of generality, we assume $k_0 = m$.
By the inductive hypothesis, there exist constants $C_1' > 0$, $C_2'>0$, 
depending on $(u_1 \ddd u_{m-1})$, 
such that
\[
\begin{gathered}
\| \tu_t - u_t \| \,\le\, C_1' \|\dt \mc{A}_1\|_{\Omega_1} \quad \text{for} \quad t = 1 \ddd m-1,
\\
\| \tu_1 \otimes \cdots \otimes \tu_{m-1} - u_1 \otimes \cdots \otimes u_{m-1} \| 
\,\le\, C_2' \|\dt \mc{A}_1\|_{\Omega_1}.
\end{gathered}
\]
By the inequality~\reff{dA1<=rho0dA0}, the above becomes
\[
\begin{gathered}
\| \tu_t - u_t \| \,\le\, C_1' \rho_0 \|\dt \mc{A}\|_{\Omega} \quad \text{for} \quad t = 1 \ddd m-1,
\\
\| \tu_1 \otimes \cdots \otimes \tu_{m-1} - u_1 \otimes \cdots \otimes u_{m-1} \| 
\,\le\, C_2' \rho_0  \|\dt \mc{A}\|_{\Omega}.
\end{gathered}
\]
In Step~5, Algorithm~\ref{alg:recursive} produces $u_{m}$ and $\tu_{m}$
by solving the linear least squares problems respectively (note that $\tomg = \Omega$)
\[
\min_{u} \, \| \mc{A} - u_1 \otimes \cdots \otimes u_{m-1} \otimes u \|_{\Omega}, \qquad
\min_{\tu} \, \| \widetilde{\mc{A}} - \tu_1 \otimes \cdots \otimes \tu_{m-1} \otimes \tu \|_{\tomg}.
\]
Similar to \reff{u2bound}, we can show that there exists $K'>0$ such that
\[
\|\tu_{m} - u_{m}\| \,\le\, K' \|\dt \mc{A}\|_{\Omega}.
\]
Hence, \reff{tu_i-u_i<=dA} holds with $C_1 \coloneqq \max\{ C_1'\rho_0, K' \}$.
Observe that
\begin{align*}
\tu_1 \otimes \cdots \otimes \tu_m - u_1 \otimes \cdots \otimes u_m 
=\,& (\tu_1 \otimes \cdots \otimes \tu_{m-1} - u_1 \otimes \cdots \otimes u_{m-1}) \otimes \tu_m \\
& + u_1 \otimes \cdots \otimes u_{m-1} \otimes (\tu_m - u_m).
\end{align*}
Similar to \reff{base}, we can show that there exists $C_2>0$, 
depending on $(u_1 \ddd u_m)$, such that \reff{tu-u<=dA} holds,
when $\|\dt\mc{A}\|_{\Omega}$ is sufficiently small.
\end{proof}

\section{Numerical Experiments}\label{sc:num_exp}

This section provides numerical experiments for solving rank one tensor completion problems with either exact or noisy observations.
All computations are implemented in MATLAB R2023b on a MacBook Pro equipped with an Apple M1 Pro processor and 16 GB of RAM. 
For neatness, only four decimal digits are displayed for computational results.

\begin{example}
Consider the p.o.t. $\mc{A} \in \re^{3 \times 3 \times 5 \times 9}$ with observed entries:
\[
\baray{cccccc}
\mc{A}_{1     2     4     4} = \frac{1}{2}, &
\mc{A}_{1     2     4     6} = \frac{1}{2}, &
\mc{A}_{1     3     2     4} = 2, &
\mc{A}_{1     3     2     7} = 2, &
\mc{A}_{1     3     4     3} = 1, &
\mc{A}_{2     1     2     2} = \frac{4}{3}, \\
\mc{A}_{2     1     2     3} = \frac{4}{3}, &
\mc{A}_{2     1     5     5} = 2, &
\mc{A}_{2     1     5     9} = 2, &
\mc{A}_{2     2     1     1} = 1, &
\mc{A}_{2     2     1     8} = 1, &
\mc{A}_{2     2     3     2} = 2, \\
\mc{A}_{2     2     3     6} = 2, &
\mc{A}_{3     3     1     7} = 3, &
\mc{A}_{3     3     1     9} = 3, &
\mc{A}_{3     3     5     1} = 9, &
\mc{A}_{3     3     5     5} = 9.
\earay
\]
Algorithm~\ref{alg:recursive} gives the exact rank one completion with decomposing vectors
\[
\baray{ccc}
u_1 = \frac{1}{\sqrt{14}} (1,2,3),
&
u_2 = \sqrt{\frac{14}{421}} (2,3,6),
&
u_3 = \sqrt{\frac{421}{3817}} (1,2,2,1,3),
\earay
\]
and 
$
u_4 = \frac{\sqrt{3817}}{6} (1,1,1,1,1,1,1,1,1).
$

\end{example}

\begin{example}
Consider the p.o.t. $\mc{A} \in \re^{3 \times 5 \times 7}$ with observed entries:
\[
\baray{l}
\mc{A}_{121} = 1.0753, 
\,\, \mc{A}_{126} = 1.0789, 
\,\, \mc{A}_{134} = 0.9170,
\,\, \mc{A}_{137} = 0.9078, 
\,\, \mc{A}_{212} = 0.9340, \\
\mc{A}_{213} = 1.0756, 
\,\, \mc{A}_{224} = 0.9197,
\,\, \mc{A}_{225} = 0.9842,
\,\, \mc{A}_{242} = 1.0916, 
\,\, \mc{A}_{251} = 1.0066, \\
\mc{A}_{257} = 1.0384,
\,\, \mc{A}_{333} = 0.9631, 
\,\, \mc{A}_{336} = 1.0373.
\earay
\]
This p.o.t. $\mc{A}$ is perturbed from the rank one tensor $\mc{A}^*$ whose entries are all ones.
Let $\dt\mc{A} \coloneqq \mc{A} - \mc{A}^*$, then
$
\|\dt\mc{A}\|_{\Omega}= 0.2382.
$
Algorithm~\ref{alg:recursive} gives the decomposing vectors:
\[
\baray{l}
u_1 = (0.5722,
0.4697,
0.6723),
\quad u_2 = (0.8721,
0.6666,
0.5456,
1.0192,
0.7603), \\
u_3 = (2.8188,
2.2802,
2.6259,
2.9373,
3.1433,
2.8283,
2.9079).
\earay
\]
The errors for the completion $u_1 \otimes u_2  \otimes u_3$ are
\[
\begin{gathered}
\|\mc{A} - u_1 \otimes u_2  \otimes u_3 \|_{\Omega} = 1.20 \cdot 10^{-15}, 
\,\quad \|\mc{A}^* - u_1 \otimes u_2  \otimes u_3 \|_{\Omega} = 0.2382, \\
\frac{\|\mc{A} - u_1 \otimes u_2  \otimes u_3 \|_{\Omega}}{\|\dt\mc{A}\|_{\Omega}} = 5.02 \cdot 10^{-15},
\quad
\frac{\|\mc{A}^* - u_1 \otimes u_2  \otimes u_3 \|_{\Omega}}{\|\dt\mc{A}\|_{\Omega}} = 1.0000.
\end{gathered}
\]
Hence, the completion error is comparable to the perturbation magnitude.
\end{example}

\begin{example}
Consider the p.o.t. $\mc{A} \in \re^{2 \times 3 \times 4 \times 5 \times 6}$ with observed entries:
\[
\baray{cccc}
\mc{A}_{11434} = 0.9993, 
& \mc{A}_{11435} = 0.9993, 
& \mc{A}_{11453} = 1.0006, 
& \mc{A}_{11455} = 0.9998, \\
\mc{A}_{12224} = 0.9993, 
& \mc{A}_{12226} = 1.0009, 
& \mc{A}_{12414} = 0.9997, 
& \mc{A}_{12421} = 1.0005, \\
\mc{A}_{12423} = 1.0005, 
& \mc{A}_{13134} = 1.0008, 
& \mc{A}_{13136} = 1.0002, 
& \mc{A}_{13143} = 1.0005, \\
\mc{A}_{13145} = 0.9997, 
& \mc{A}_{13232} = 0.9995, 
& \mc{A}_{13233} = 1.0008, 
& \mc{A}_{13241} = 0.9999, \\
\mc{A}_{13244} = 1.0009, 
& \mc{A}_{21211} = 1.0003, 
& \mc{A}_{21216} = 1.0002, 
& \mc{A}_{21251} = 0.9992, \\
\mc{A}_{21255} = 1.0009, 
& \mc{A}_{21334} = 0.9999, 
& \mc{A}_{21336} = 1.0002, 
& \mc{A}_{21344} = 0.9998, \\
\mc{A}_{21345} = 0.9995, 
& \mc{A}_{23133} = 1.0008, 
& \mc{A}_{23135} = 1.0001, 
& \mc{A}_{23151} = 0.9990, \\
\mc{A}_{23152} = 1.0002, 
& \mc{A}_{23334} = 0.9997, 
& \mc{A}_{23336} = 1.0001, 
& \mc{A}_{23351} = 1.0008, \\
\mc{A}_{23353} = 0.9997.
\earay
\]
This p.o.t. $\mc{A}$ is perturbed from the rank one tensor $\mc{A}^*$ whose entries are all ones.
Let $\dt\mc{A} \coloneqq \mc{A} - \mc{A}^*$, then
$
\|\dt\mc{A}\|_{\Omega}=3.2 \cdot 10^{-3}.
$
Algorithm~\ref{alg:recursive} gives the decomposing vectors:
\[
\baray{l}
u_1 = (0.7069,
0.7073),
\quad u_2 = (0.6323,
0.6320,
0.6329), \\
u_3 = (0.7451,
0.7452,
0.7450,
0.7458), \\
u_4 = (0.7276,
0.7283,
0.7274,
0.7273,
0.7276), \\
u_5 = (4.1220,
4.1228,
4.1240,
4.1236,
4.1229,
4.1247).
\earay
\]
The errors for the completion $u_1 \otimes \cdots  \otimes u_5$ are
\[
\begin{gathered}
\|\mc{A} - u_1 \otimes \cdots \otimes u_5 \|_{\Omega} = 3.3  \cdot 10^{-3},
\quad \|\mc{A}^* - u_1 \otimes \cdots \otimes u_5 \|_{\Omega} = 2.2  \cdot 10^{-3}, \\
\frac{\|\mc{A} - u_1 \otimes \cdots \otimes u_5 \|_{\Omega}}{\|\dt\mc{A}\|_{\Omega}} = 1.0277,
\quad\qquad \,\,\frac{\|\mc{A}^* - u_1 \otimes \cdots \otimes u_5 \|_{\Omega}}{\|\dt\mc{A}\|_{\Omega}} = 0.6853.
\end{gathered}
\]
Hence, the completion error is comparable to the perturbation magnitude.
\end{example}

In the following, 
we give a method to randomly generate the observation set $\Omega$ for which $(\mc{A}, \Omega)$ is rank one determinable.
First, let $\Phi_1 \coloneqq [n_1]$.
For $t = 1 \ddd m-1$, we construct 
$
\Phi_{t+1} \subseteq [n_1] \times \cdots \times [n_{t+1}]
$
recursively as follows.
We write
\be
\Phi_t \,\coloneqq\, \{  \phi_1 \ddd \phi_{M} \},
\ee
where $M \coloneqq |\Phi_t|$ is the cardinality of $\Phi_t$.
\bit
\item For the case $M \ge n_{t+1}$,
we choose a random permutation $( i_1 \ddd i_M )$ of $(1 \ddd M)$ and 
a random permutation $(j_1 \ddd j_{n_{t+1}})$ of $(1 \ddd n_{t+1})$.
For $l = n_{t+1}+1 \ddd M$, we select $j_l \in [n_{t+1}]$ randomly.
Let
\be
\Phi_{t+1} \,\coloneqq\, \{(\phi_{i_1}, j_1)\} \cup 
\bigcup_{l=2}^{M} \big\{ (\phi_{i_l}, j_{l-1}), (\phi_{i_l}, j_l) \big\}.
\ee

\item For the case $M < n_{t+1}$,
we choose a random permutation $( i_1 \ddd i_M )$ of $(1 \ddd M)$ and 
a random permutation $( j_1 \ddd j_{n_{t+1}} )$ of $(1 \ddd n_{t+1})$.
For $l = M+1 \ddd n_{t+1}$, we select $i_l \in [M]$ randomly.
Let
\be
\Phi_{t+1} \,\coloneqq\, \{(\phi_{i_1}, j_1)\} \cup 
\bigcup_{l=2}^{n_{t+1}} \big\{ (\phi_{i_l}, j_{l-1}), (\phi_{i_l}, j_l) \big\}.
\ee
\eit
Let $N \coloneqq \max\{ M, n_{t+1} \}$. 
Then $\Phi_{t+1}$ represents the edge set of the bipartite graph:
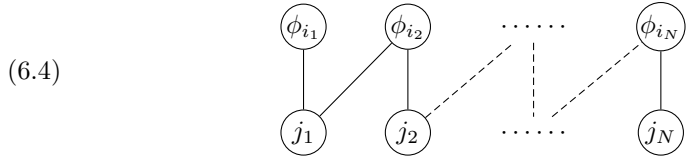
\be\label{con_path:phi-j}
\begin{tikzcd}[column sep=2.2em,row sep=2.2em]
|[draw,circle,minimum size=6mm,inner sep=0pt]| {\phi_{i_1}}
& |[draw,circle,minimum size=6mm,inner sep=0pt]| {\phi_{i_2}}
& \cdots\cdots
& |[draw,circle,minimum size=6mm,inner sep=0pt]| {\phi_{i_N}} \\
|[draw,circle,minimum size=6mm,inner sep=0pt]| {j_1}
& |[draw,circle,minimum size=6mm,inner sep=0pt]| {j_2}
& \cdots\cdots
& |[draw,circle,minimum size=6mm,inner sep=0pt]| {j_N}
\arrow[no head, from=1-1, to=2-1]
\arrow[no head, from=1-2, to=2-2]
\arrow[dashed, no head, from=1-3, to=2-3]
\arrow[no head, from=1-4, to=2-4]
\arrow[no head, from=2-1, to=1-2]
\arrow[dashed, no head, from=2-2, to=1-3]
\arrow[dashed, no head, from=2-3, to=1-4]
\end{tikzcd}
\ee
As shown above, the bipartite graph $G(\Phi_t, [n_{t+1}], \Phi_{t+1})$ is connected.
By Theorem~\ref{thm:connect<->rank-1}(ii), the p.o.m. $(A, \Phi_{t+1})$ is rank one extractable for any rank one matrix $A$ without zero entries.
Finally, let
\be
\Omega \coloneqq \Phi_m \subseteq [n_1] \times \cdots \times [n_m].
\ee

We randomly generate vectors $u_1 \in \re^{n_1} \ddd u_m \in \re^{n_m}$ whose entries obey the normal distribution $\mc{N}(0,1)$,
and let
$
\mc{A} = u_1 \otimes \cdots \otimes u_m.
$
By the above construction of $\Omega$, $(\mc{A}, \Omega)$ is rank one determinable.
Let $\widetilde{\mc{A}} \coloneqq \mc{A} + \dt\mc{A}$ be a perturbation of $\mc{A}$ on $\Omega$.
For every $(i_1 \ddd i_m) \in \Omega$, we generate the noise 
\[
\dt\mc{A}_{i_1 \ddd i_m} \,=\, \vareps \cdot (\mt{2*rand-1}) \cdot \mc{A}_{i_1 \ddd i_m},
\]
where $\vareps$ is the magnitude of noise (e.g., $\vareps = 10^{-2}$).
Let $\tu_1 \otimes \cdots \otimes \tu_m$ be the rank one completion produced by Algorithm~\ref{alg:recursive} from $(\widetilde{\mc{A}}, \Omega)$.
We evaluate the completion quality by the errors
\[
\mt{err\_ab} = \|u_1 \otimes \cdots \otimes u_m - \tu_1 \otimes \cdots \otimes \tu_m \|_{\Omega},
\qquad \mt{err\_rt} = \frac{\mt{err\_ab}}{\|\dt\mc{A}\|_{\Omega}}.
\]
The difference between decomposing vectors is measured by the angle $\theta$:
\[
\theta_k = \cos^{-1} \Big( \frac{u_k^T\tu_k}{\|u_k\| \|\tu_k\|} \Big), \
k = 1 \ddd m,
\qquad \sin\theta = \frac{1}{m} \sum_{k=1}^m |\sin \theta_k|.
\]
For the p.o.t. $(\mc{A}, \Omega)$, the density of observed entries is
\[
\mt{den} \,=\, |\Omega| / (n_1 \cdots n_m).
\]

We report computational results on randomly generated p.o.t.'s as above.
For each case of $(n_1,\ldots,n_m)$, 
we generate $50$ independent instances of $(\widetilde{\mc{A}},\Omega)$.
Algorithm~\ref{alg:recursive} is applied to obtain a rank one completion
$\tu_1 \otimes \cdots \otimes \tu_m$. 
We report in Table~\ref{tab:m=3,4,5,6} the average values (over 50 instances)
of $\mt{den}$,  $\mt{err\_ab}$, $\mt{err\_rt}$, 
$\sin\theta$, and the runtime (in seconds).
These results indicate that our method is robust under noise.
The values of $\mt{err\_rt}$ are close to 1 and $\sin\theta$ remains small,
suggesting that the completion error is comparable to the perturbation magnitude.
This aligns with Theorem~\ref{thm:tensor_perturb}.
The runtime shows that our method is efficient for large-scale problems.

\begin{table}[htb]
\caption{Computational results for noisy rank one tensor completion.}
\label{tab:m=3,4,5,6}
\centering
\begin{tabular}{ccccccc}
\toprule
$(n_1 \ddd n_m)$ & $\mt{den}$  & $\vareps$ & $\mt{err\_ab}$ & $\mt{err\_rt}$ & $\sin\theta$ & $\mt{time}$ \\
\midrule
\multirow{2}{*}{$(700,800,900)$} &
\multirow{2}{*}{$6.3 \cdot 10^{-6}$} &
$10^{-2}$ & 0.0366 & 1.01 & 0.026 & 10.48 \\
& & $10^{-3}$ & 0.0036 & 1.00 & 0.002 & 10.19 \\
\midrule
\multirow{2}{*}{$(800,900,1000)$} &
\multirow{2}{*}{$5.0 \cdot 10^{-6}$} &
$10^{-2}$ & 0.0371 & 1.00 & 0.026 & 15.34 \\
& & $10^{-3}$ & 0.0038 & 1.00 & 0.002 & 14.74 \\
\midrule
\multirow{2}{*}{$(250,300,350,400)$} &
\multirow{2}{*}{$ 2.3 \cdot 10^{-7}$} &
$10^{-2}$ & 0.0289 & 1.12 & 0.031 & 12.99 \\
& & $10^{-3}$ & 0.0026 & 1.04 & 0.002 & 13.24 \\
\midrule
\multirow{2}{*}{$(300,350,400,450)$} &
\multirow{2}{*}{$ 1.5 \cdot 10^{-7}$} &
$10^{-2}$ & 0.0283 & 1.07 & 0.029 & 22.13 \\
& & $10^{-3}$ & 0.0028 & 1.03 & 0.002 & 22.08 \\
\midrule
\multirow{2}{*}{$(80,100,120,140,160)$} &
\multirow{2}{*}{$ 7.3 \cdot 10^{-8}$} &
$10^{-2}$ & 0.0225 & 1.21 & 0.029 & 9.64 \\
& & $10^{-3}$ & 0.0021 & 1.11 & 0.002 & 9.63 \\
\midrule
\multirow{2}{*}{$(100,120,140,160,180)$} &
\multirow{2}{*}{$ 3.9 \cdot 10^{-8}$} &
$10^{-2}$ & 0.0234 & 1.19 & 0.031 & 18.99 \\
& & $10^{-3}$ & 0.0024 & 1.10 & 0.002 & 18.70 \\
\midrule
\multirow{2}{*}{$(30, 35, 40, 45, 50, 55)$} &
\multirow{2}{*}{$ 2.1 \cdot 10^{-7}$} &
$10^{-2}$ & 0.0174 & 1.12 & 0.019 & 10.95 \\
& & $10^{-3}$ & 0.0019 & 1.13 & 0.002 & 10.98 \\
\midrule
\multirow{2}{*}{$(35, 40, 45, 50, 55, 60)$} &
\multirow{2}{*}{$ 1.2 \cdot 10^{-7}$} &
$10^{-2}$ & 0.0185 & 1.10 & 0.017 & 18.17 \\
& & $10^{-3}$ & 0.0019 & 1.14 & 0.002 & 18.48 \\
\bottomrule
\end{tabular}
\end{table}

We further compare Algorithm~\ref{alg:recursive} with the direct nonlinear least squares approach,
which can be formulated as
\be\label{nls}
\min_{\tu_1 \ddd \tu_m} \quad 
\sum_{(i_1 \ddd i_m) \in \Omega}
\left( \widetilde{\mc A}_{i_1 \cdots i_m} - (\tu_1)_{i_1} \cdots (\tu_m)_{i_m} \right)^2.
\ee
For convenience in implementation, 
we use the MATLAB function $\mt{lsqnonlin}$ with a random initialization to solve \reff{nls}.
For each case of $(n_1 \ddd n_m)$, 
we generate 20 independent instances of $(\widetilde{\mc{A}},\Omega)$ as above.
We report in Table~\ref{tab:compare_nls} the computational results,
which indicate that our method is not only more accurate 
but also more efficient than the classical nonlinear least squares approach.

\begin{table}[htb]
\caption{Comparison with nonlinear least squares.}
\label{tab:compare_nls}
\centering
\begin{tabular}{cccrrr}
\toprule
\multirow{2}{*}{$(n_1 \ddd n_m)$}
& \multirow{2}{*}{$\vareps$} 
& \multicolumn{2}{c}{Algorithm~\ref{alg:recursive}}
& \multicolumn{2}{c}{Nonlinear L.S.} \\
\cmidrule(lr){3-4} \cmidrule(lr){5-6}
&& $\mt{err\_rt}$ & $\mt{time}$
& $\mt{err\_rt}$ & $\mt{time}$ \\
\midrule
$(400, 500, 600)$ & $10^{-2}$ & 1.02 & 1.91 & 62.87 & 23.24 \\
\midrule
$(500, 600, 700)$ & $10^{-2}$ & 1.01 & 3.53 & 68.85 & 36.83 \\
\midrule
$(600, 700, 800)$ & $10^{-2}$ & 1.01 & 5.65 & 68.41 & 56.01 \\
\midrule
$(150, 200, 250, 300)$ & $10^{-2}$ & 1.09 & 2.64 & 56.88 & 16.72 \\
\midrule
$(200, 250, 300, 350)$ & $10^{-3}$ & 1.06 & 5.68 & 743.97 & 25.31 \\
\midrule
$(250, 300, 350, 400)$ & $10^{-2}$ & 1.06 & 10.65 & 66.33 & 43.01 \\
\midrule
$(40, 60, 80, 100, 120)$ & $10^{-3}$ & 1.27 & 1.15 & 613.69 & 4.64 \\
\midrule
$(60, 80, 100, 120, 140)$ & $10^{-3}$ & 1.15 & 3.31 & 705.51 & 7.81 \\
\bottomrule
\end{tabular}
\end{table}

\section{Conclusions}

This paper studies the rank one completion problem for tensors of arbitrary orders.
We introduce rank one extractability for partially observed matrices 
and study its properties.
We then define rank one determinable tensors recursively through flattenings.
Based on this, a recursive algorithm is proposed for rank one tensor completion. 
Under some assumptions, 
we prove that this algorithm produces a unique rank one completion.
In the presence of noise, we show that this algorithm is robust.
That is, when the noise is sufficiently small, 
the rank one completion produced by the algorithm is close to the exact one.
Numerical experiments illustrate the efficiency and accuracy of our method.

\smallskip
\noindent \textbf{Acknowledgements.}
Jiawang Nie and Linghao Zhang are partially supported by 
the NSF grant DMS2513254 
and the AFOSR grant FA9550-25-1-0298.

\appendix

\section{Perturbation of matrix nullspace}

\begin{theorem}\label{thm:mx_null_perturb}
Let $B \in \re^{\ell \times n}$ be a matrix with $\rank (B) = n-1$.
Let $v$ be a solution to $Bx = 0$ with $\|v\|=1$.
Let $\widetilde{B} = B + \dt B$ be a perturbation of $B$.

\bnum
\item[(i)] If $\ell \ge n$, let $\tv$ be the right singular vector of $\widetilde{B}$ corresponding to the minimum singular value $\tsig_n$ such that $\|\tv\|=1$ and $v^T\tv \ge 0$.

\item[(ii)] If $\ell<n$, let $\tv$ be a solution to $\widetilde{B}x = 0$ with $\|\tv\|=1$ and $v^T\tv \ge 0$.
\enum
Let $\dt v \coloneqq \tv - v$.
Then there exists a constant $C_0>0$, depending on $B$, such that
\be\label{dv<=C0dB}
\|\dt v\| \,\le\, C_0 \|\dt B\|,
\ee
when $\|\dt B\|$ is sufficiently small.
\end{theorem}

\begin{proof}
Let $B = U\Sig V^T$ be the SVD with 
\[
\Sig = 
\bbm 
\Sig_1 & 0 \\
0 & 0
\ebm
\in \re^{\ell \times n},
\quad \Sig_1 = \diag(\sig_1 \ddd \sig_{n-1}),
\]
where $\sig_1 \ge \cdots \ge \sig_{n-1} > \sig_n = 0$ are singular values of $B$, and
\[
U = \bbm u_1 & u_2 & \cdots & u_{\ell}  \ebm \in \re^{\ell \times \ell},
\quad
V = \bbm v_1 & v_2 & \cdots & v_n  \ebm \in \re^{n \times n}
\]
are orthogonal matrices.
Since $\rank(B) = n-1$, the nullspace of $B$ is one-dimensional,
so the vector $v_n$ is unique up to sign.
We can generally assume $v_n = v$.

\noindent (i) For the case that $\ell \ge n$, 
up to orthogonal transformation, we can assume $B=\Sig$ and $v=e_n$.
Let $\tu$ be the unit length left singular vector of $\widetilde{B}$, corresponding to its minimum singular value $\tsig_n$,
i.e., $\widetilde{B} \tv = \tsig_n \tu$.
Note that
\be\label{tBtvn}
\widetilde{B} \tv = (B +\dt B) (v + \dt v) = Bv + B \cdot \dt v + \dt B \cdot v + \dt B \cdot \dt v.
\ee
Let $y \coloneqq B\cdot \dt v$.
Since $Bv = 0$ and $\widetilde{B} \tv = \tsig_n \tu$, the above implies
\[
y = \tsig_n \tu - \dt B \cdot v - \dt B \cdot \dt v.
\]
Write
$
y = 
[ y_1 \,\, y_2 \,\, \cdots \,\, y_{\ell} ]^T.
$
Since $\|\tu\| = \|v\| = \|\tv\| = 1$ and $\|\dt v\| = \|\tv - v\| \le 2$, 
\[
\|y\| \le \tsig_n \|\tu\| + \|\dt B\| \|v\| + \|\dt B\| \|\tv - v\| \le \tsig_n + 3\|\dt B\|.  
\]
Note $\tsig_n$ (resp., $\sig_n$) is the minimum singular value of $\widetilde{B}$ (resp., $B$),
so
\begin{align*}
\tsig_n &= \min_{\|x\|=1} \|\widetilde{B}x\| \le \min_{\|x\|=1}( \|Bx\| + \|\dt B \cdot x\| ) \\
&\le \min_{\|x\|=1} \|Bx\| + \|\dt B\| = \sig_n + \|\dt B\|.
\end{align*}
Since $\rank(B) = n-1$, we have $\sig_n = 0$.
This implies
$
\tsig_n \le \|\dt B\|
$,
so $\|y\| \le 4\|\dt B\|$.
Denote
$[ z_1 \,\, z_2 \,\, \cdots \,\, z_n ] = \dt v^T$.
Since
$
y = B\cdot \dt v = \Sig \cdot \dt v,
$
we get $z_i = y_i / \sig_i$ for $i \in [n-1]$.
Let 
$
\hat{z} = [z_1  \,\, \cdots \,\, z_{n-1} ]^T,
$
then
\[
\|\hat{z}\|^2 = \sum_{i=1}^{n-1} \frac{y_i^2}{\sig_i^2} 
\le \frac{1}{\sig_{n-1}^2} \sum_{i=1}^{n-1} y_i^2
\le \frac{\|y\|^2}{\sig_{n-1}^2}.
\]
Since
$
\tv = v + \dt v = e_n + \dt v = [ z_1  \,\, \cdots \,\, z_{n-1} \,\, 1 + z_n ]^T$,
we have
\[
1 = \|\tv\|^2 = \|\hat{z}\|^2 + (1+z_n)^2,
\quad 1+z_n = \tv^Te_n = \tv^Tv \ge 0.
\]
Hence, $z_n \le 0$ and
\be\label{z_n<=||zhat||^2}
z_n = \sqrt{ 1 - \|\hat{z}\|^2 } - 1 = -\|\hat{z}\|^2 \Big/ \big(\sqrt{ 1 - \|\hat{z}\|^2 } + 1\big).
\ee
This implies $|z_n| \le \|\hat{z}\|^2$, so
\[
\|\dt v\|^2 = \|\hat{z}\|^2 + z_n^2 \le \|\hat{z}\|^2 + \|\hat{z}\|^4 \le
\frac{\|y\|^2}{\sig_{n-1}^2} \Big( 1 + \frac{\|y\|^2}{\sig_{n-1}^2} \Big).
\]
Recall that $\|y\| \le 4\|\dt B\|$.
If $\|\dt B\| \le \frac{\sig_{n-1}}{4}$, then
\[
\|y\| \,\le\, \sig_{n-1} 
\quad \implies \quad
1 + \|y\|^2 \big/ \sig_{n-1}^2 \,\le\, 2,
\]
so
$\|\dt v\| \le \frac{4\sqrt{2}}{\sig_{n-1}} \|\dt B\|$.
Hence, \reff{dv<=C0dB} holds for 
$C_0 =  \frac{4\sqrt{2}}{\sig_{n-1}}$ 
if $\|\dt B\| \le \frac{\sig_{n-1}}{4}$.

\noindent (ii) For the case that $\ell < n$, we have $\ell = n-1$ since $\rank(B) = n-1$.
In the SVD $B = U\Sig V^T$, we have 
$
\Sig = 
[ \Sig_1 \,\, 0 ] \in \re^{(n-1) \times n}.
$
Since $\widetilde{B}\tv = 0$ with $\|\tv\|=1$, \reff{tBtvn} gives
\[
U\Sig V^T \dt v = B \cdot \dt v = - \dt B \cdot v - \dt B \cdot \dt v.
\]
Multiplying $U^T$ on both sides, we have
\[
\Sig V^T  \dt v = - U^T  \dt B \cdot v - U^T \dt B \cdot \dt v.
\]
Let 
$
y = [ y_1 \,\, y_2 \,\, \cdots \,\, y_{n-1} ]^T = \Sig V^T \dt v.
$
Since $U$ is orthogonal and $\|v\| = \|\tv\| = 1$,
\[
\|y\| \le \|U^T\| \|\dt B\| \|v\| + \|U^T\| \|\dt B\| \|\tv - v\| \le 3\|\dt B\|.
\]
Let $z = [ z_1 \,\, z_2 \,\, \cdots \,\, z_n ]^T = V^T  \dt v$
and
$
\hat{z} = [ z_1 \,\, \cdots \,\, z_{n-1} ]^T,
$
then $y = \Sig z$ implies
\[
\|\hat{z}\|^2 \,\le\, \|y\|^2 \big/ \sig_{n-1}^2.
\]
Since
$
V^T\tv = V^T(v+\dt v) = e_n + z = [ z_1 \,\, \cdots \,\, z_{n-1} \,\, 1 + z_n ]^T,
$
we have
\[
1 = \|\tv\|^2 = \|V^T\tv\|^2 = \|\hat{z}\|^2 + (1+z_n)^2,
\]
\[
1+z_n = (V^T\tv)^Te_n = \tv^TVe_n = \tv^Tv \ge 0.
\]
Similar to \reff{z_n<=||zhat||^2}, one can show $|z_n| \le \|\hat{z}\|^2$ and
\[
\|\dt v\|^2 = \|V^T \dt v\|^2 =
\|\hat{z}\|^2 + z_n^2 
\le \frac{\|y\|^2}{\sig_{n-1}^2} \Big( 1 + \frac{\|y\|^2}{\sig_{n-1}^2} \Big).
\]
Recall that $\|y\| \le 3\|\dt B\|$.
If $\|\dt B\| \le \frac{\sig_{n-1}}{3}$, then
$
\|\dt v\| \le \frac{3\sqrt{2}}{\sig_{n-1}} \|\dt B\|.
$
Hence, \reff{dv<=C0dB} holds for 
$C_0 =  \frac{3\sqrt{2}}{\sig_{n-1}}$ 
if $\|\dt B\| \le \frac{\sig_{n-1}}{3}$.
\end{proof}


\begin{thebibliography}{100}

\bibitem{ashraphijuo2017}
{\sc M. Ashraphijuo, V. Aggarwal, and X. Wang}, 
{\em A characterization of sampling patterns for low-Tucker-rank tensor completion problem}, 
2017 IEEE International Symposium on Information Theory, Aachen, 2017, pp. 531--535.

\bibitem{bai2016}
{\sc M. Bai, X. Zhang, G. Ni, and C. Cui}, 
{\em An adaptive correction approach for tensor completion}, 
SIAM J. Imaging Sci., 
9(3), pp. 1298--1323, 2016.

\bibitem{barak2022}
{\sc B. Barak and A. Moitra}, 
{\em Noisy tensor completion via the sum-of-squares hierarchy}, 
Math. Program., 
193, pp. 513--548, 2022.

\bibitem{cifuentes25}
{\sc D.~Cifuentes and Z.~Li},
{\em Solving exact and noisy rank-one tensor completion with semidefinite programming},
arXiv preprint arXiv:2511.06062, 2025.

\bibitem{cosse2021}
{\sc A.~Cosse and L.~Demanet},
{\em Stable rank-one matrix completion is solved by the level 2 Lasserre relaxation},
Found. Comput. Math., 
21, 891--940 (2021). 

\bibitem{LMV2000}
{\sc L.~De~Lathauwer, B.~De~Moor, and J.~Vandewalle},
{\em A multilinear singular value decomposition},
SIAM J. Matrix Anal. Appl., 
21 (2000), pp. 1253--1278.

\bibitem{demmel97}
{\sc J.~W.~Demmel},
{\em Applied numerical linear algebra},
SIAM, 1997.

\bibitem{dong2022}
{\sc S. Dong, B. Gao, Y. Guan, and F. Glineur}, 
{\em New Riemannian preconditioned algorithms for tensor completion via polyadic decomposition}, 
SIAM J. Matrix Anal. Appl., 
43, pp. 840--866, 2022.

\bibitem{Frolov2015}
{\sc E. Frolov and I. Oseledet}, 
{\em Tensor methods and recommender systems}, 
WIRES. DATA. MIN. KNOWL., 
7, p. e1201, 2015.

\bibitem{harshman1970}
{\sc  R. A. Harshman},
{\em Foundations of the PARAFAC procedure: Models and conditions for an “explanatory” multi-modal factor analysis}, 
UCLA working papers in phonetics, 16, pp. 1--84, 1970.

\bibitem{Hillar2013}
{\sc C. Hillar and L.-H. Lim}, 
{\em Most tensor problems are NP-hard}, 
J. ACM., 
60, pp. 1--39, 2013.

\bibitem{hitchcock1927}
{\sc F. L. Hitchcock}, 
{\em The expression of a tensor or a polyadic as a sum of products}, 
J. Math. Phys., 
6, pp. 164--189, 1927.

\bibitem{hitchcock1928}
{\sc F.~L. Hitchcock},  
{\em Multiple invariants and generalized rank of a $p$-way matrix or tensor}, 
J. Math. Phys., 
7 (1928), pp.~39--79.

\bibitem{jiang2019}
{\sc Q. Jiang and M. Ng}, 
{\em Robust low-tubal-rank tensor completion via convex optimization},
IJCAI, pp. 2649--2655, 2019.

\bibitem{kahle17}
{\sc T. Kahle, K. Kubjas, M. Kummer, and Z. Rosen}, 
{\em The geometry of rank-one tensor completion}, 
SIAM J. Appl. Algebra. Geom., 
1 (2017), pp. 200--221.

\bibitem{Karatzoglou2010}
{\sc A. Karatzoglou, X. Amatriain, L. Baltrunas, and N. Oliver}, 
{\em Multiverse recommendation: n-dimensional tensor factorization for context-aware collaborative filtering}, 
Proceedings of the 4th ACM Conference on Recommender Systems, 
New York, NY, 2010, pp. 79--86.

\bibitem{kilmer2011}
{\sc M. E. Kilmer and C. D. Martin}, 
{\em Factorization strategies for third-order tensors}, 
Linear Algebra and its Applications, 
435 (2011), pp. 641--658.

\bibitem{kiraly12}
{\sc F. Kir\'{a}ly, L.~Theran, and R. Tomioka},
{\em The algebraic combinatorial approach for low-rank matrix completion},
J. Mach. Learn. Res., 
16(1):1391--1436, 2015.

\bibitem{KolBad09}
{\sc T.~Kolda and B.~W.~Bader},
{\em Tensor decompositions and applications},
SIAM Rev., 
51 (2009), pp.~455--500.

\bibitem{kressner2014}
{\sc D. Kressner, M. Steinlechner, and B. Vandereycken}, 
{\em Low-rank tensor completion by Riemannian optimization}, 
BIT Numer. Math., 
54 (2014), pp. 447--468.

\bibitem{Lim13}
{\sc L.-H.~Lim},
{\em Tensors and hypermatrices, in: L. Hogben (Ed.)},
Handbook of linear algebra, 2nd Ed.,
CRC Press, Boca Raton, FL, 2013.

\bibitem{Lim2010}
{\sc L.-H. Lim and P. Comon}, 
{\em Multiarray signal processing: tensor decomposition meets compressed sensing}, 
Comptes Rendus Mecanique, 338(6), 311--320.

\bibitem{Lim2013}
{\sc L.-H. Lim and P. Comon}, 
{\em Blind multilinear identification}, 
IEEE. Trans. Inf. Theory, 60
(2013), pp. 1260--1280.

\bibitem{mu2014}
{\sc C. Mu, B. Huang, J. Wright, and D. Goldfarb}, 
{\em Square deal: lower bounds and improved relaxations for tensor recovery}, 
Proceedings of the International Conference on Machine Learning (PMLR), 
32 (2014), pp. 73--81.

\bibitem{nieGP}
{\sc J.~Nie}, 
{\em Generating polynomials and symmetric tensor decompositions},
Found. Comput. Math., 
17, 423--465 (2017). 

\bibitem{nieSymApprox}
{\sc J.~Nie}, 
{\em Low rank symmetric tensor approximations},
SIAM J. Matrix Anal. Appl., 
38(4), 1517--1540, 2017.

\bibitem{nieNuclear}
{\sc J.~Nie}, 
{\em Symmetric tensor nuclear norms}, 
SIAM J. Appl. Algebra. Geom., 
1(1), 599--625, 2017.

\bibitem{nie25}
{\sc J. Nie, X. Tang, and J. Zhou},
{\em Robust completion for rank-1 tensors with noises},
arXiv preprint arXiv:2504.00398, 2025.

\bibitem{qin2022}
{\sc W.~Qin, H.~Wang, F.~Zhang, J.~Wang, X.~Luo, and T.~Huang}, 
{\em Low-rank high-order tensor completion with applications in visual data}, 
IEEE. Trans. Image. Process.,
31 (2022), pp.~2433--2448.

\bibitem{qiu2021}
{\sc D.~Qiu, M.~Bai, M.~K. Ng, and X.~Zhang}, 
{\em Robust low transformed multi-rank tensor methods for image alignment},
J. Sci. Comput., 
87 (2021), Article No.~24.

\bibitem{Singh20}
{\sc M.~Singh, A.~Shapiro, and R.~Zhang},
{\em Rank one tensor completion problem},
arXiv preprint arXiv:2009.10533, 2020.

\bibitem{steinlechner2016}
{\sc M. Steinlechner}, 
{\em Riemannian optimization for high-dimensional tensor completion}, 
SIAM J. Sci. Comput., 
38 (2016), pp. S461--S484.

\bibitem{swijsen2022}
{\sc L. Swijsen, J. van der Veken, N. Vannieuwenhoven}, 
{\em Tensor completion using geodesics on Segre manifolds}, 
Numer. Linear Algebra Appl., 
29(2022), e2446.

\bibitem{tang2015}
{\sc G. Tang and P. Shah}, 
{\em Guaranteed tensor decomposition: A moment approach},
International Conference on Machine Learning, Lille, FR, 37 (2015), pp. 1491--1500.

\bibitem{tian2024}
{\sc F. Tian, M. Pasha, M. E. Kilmer, E. Miller, and A. Patra}, 
{\em Tensor completion with BMD factor nuclear norm minimization}, 
arXiv preprint arXiv:2402.13068, 2024.

\bibitem{wang19}
{\sc A. Wang, Z. Lai, and Z. Jin}, 
{\em Noisy low-tubal-rank tensor completion}, 
Neurocomputing, 330 (2019), pp. 267--279.

\bibitem{yuan2016}
{\sc M. Yuan and C.-H. Zhang}, 
{\em On tensor completion via nuclear norm minimization}, 
Found. Comput. Math., 
16 (2016), pp. 1031--1068.

\bibitem{zhang2017}
{\sc Z. Zhang and S. Aeron}, 
{\em Exact tensor completion using t-SVD}, 
IEEE Transactions on Signal Processing, 65 (2017), pp. 1511--1526.

\bibitem{zhao2020}
{\sc X. Zhao, M. Bai, and M. Ng}, 
{\em Nonconvex optimization for robust tensor completion from grossly sparse observations}, 
J. Sci. Comput., 
85 (2020), 46.

\bibitem{zhou25}
{\sc J. Zhou, J. Nie, Z. Peng, and G. Zhou},
{\em The rank-1 completion problem for cubic tensors},
SIAM J. Matrix Anal. Appl., 46(1):151--171, 2025.

\end{thebibliography}
\end{document}